\documentclass[onefignum,onetabnum,backref,twoside]{siamart190516}

\usepackage{color}
\usepackage{lipsum}
\usepackage{amsfonts}
\usepackage{graphicx}
\usepackage{caption}
\usepackage{subcaption}
\usepackage{epstopdf}
\usepackage{calc}
\ifpdf
  \DeclareGraphicsExtensions{.eps,.pdf,.png,.jpg}
\else
  \DeclareGraphicsExtensions{.eps}
\fi

\newsiamremark{remark}{Remark}
\newsiamremark{hypothesis}{Hypothesis}
\crefname{hypothesis}{Hypothesis}{Hypotheses}
\newsiamthm{claim}{Claim}

\headers{Tree quasi-separable matrices}{Govindarajan, Chandrasekaran, Dewilde}

\title{Tree quasi-separable matrices: a simultaneous generalization of sequentially and hierarchically semi-separable representations
}

\author{Nithin Govindarajan\thanks{{Postdoctoral researcher, Electrical Eng. (ESAT), KU Leuven, Leuven, Belgium
  (\email{nithin.govindarajan@kuleuven.be})}}
  \and Shivkumar Chandrasekaran\thanks{Professor, Electrical and Computer Eng., University of California Santa Barbara, CA, United States
  (\email{shiv@ucsb.edu})} 
  \and
  Patrick Dewilde \thanks{Professor Emeritus, Electrical Engineering, Mathematics and Computer Science, TU Delft, Delft, The Netherlands and Institute of Advanced Study, Technical University Munich, Germany.
  (\email{p.dewilde@me.com})} 
  }

\usepackage{amsopn}

\ifpdf
\hypersetup{
  pdftitle={Macaulay},
  pdfauthor={N. Govindarajan, R. Widdershoven, S. Chandrasekaran, L. De Lathauwer}
}
\fi

\usepackage{ stmaryrd }
\usepackage{enumerate}
\usepackage{amssymb}
\usepackage{mathrsfs}
\usepackage{mathabx}
\usepackage{arydshln}
\usepackage{mathdots}
\usepackage{enumitem}
\usepackage[nocompress]{cite}
\usepackage{dsfont}
\usepackage{bm}
\usepackage{soul}
\usepackage{bbold}
\usepackage{blkarray}
\usepackage{booktabs, bigstrut}
\usepackage{caption}
\usepackage{subcaption}
\usepackage{rotating}
\usepackage{subcaption}
\usepackage{algpseudocode}
\usepackage{booktabs}
\usepackage{multirow}

\usepackage{xcolor}
\definecolor{SVDcol}{RGB}{0,0,0}
\definecolor{GECPMcol}{RGB}{199,240,160}
\definecolor{GECPCcol}{RGB}{74,226,217}
\definecolor{GEAPCcol}{RGB}{87,135,255}

\newtheorem{alg}{Algorithm}
\newlength{\alginwidth}
\newlength{\algoutwidth}
\newcommand{\field}{\mathbb{F}}

\newcommand{\naturals}{\mathbb{N}}

\DeclareMathOperator{\rank}{rank}

\DeclareMathOperator{\degnode}{deg}
\newcommand{\Id}{\mathrm{Id}}

\newcommand{\mat}[1]{\mathrm{#1}}
\newcommand{\vect}[1]{\bm{#1}}

\newcommand{\Graph}{\mathbb{G}}

\newcommand{\kNeighbor}{\mathcal{N}}
\newcommand{\NodeSet}{\mathbb{V}}
\newcommand{\EdgeSet}{\mathbb{E}}
\newcommand{\EdgeCount}{\mathcal{E}}
\newcommand{\nodeid}[1]{\lbrace#1\rbrace}
\newcommand{\Spinner}{\mat{S}}
\newcommand{\Dmat}{\mat{D}}
\newcommand{\Bmat}{\mat{B}}
\newcommand{\Pmat}{\mat{P}}
\newcommand{\Umat}{\mat{U}}
\newcommand{\Vmat}{\mat{V}}
\newcommand{\Wmat}{\mat{W}}
\newcommand{\Qmat}{\mat{Q}}
\newcommand{\Cmat}{\mat{C}}
\newcommand{\Inp}{\mat{Inp}}
\newcommand{\Trans}{\mat{Trans}}
\newcommand{\Out}{\mat{Out}}
\newcommand{\Level}{\mathcal{L}}
\newcommand{\Path}{\mathbb{P}}
\newcommand{\Parent}{\mathcal{P}}
\newcommand{\Children}{\mathcal{C}}
\newcommand{\Siblings}{\mathcal{S}}
\newcommand{\TQS}{\mat{T}}
\newcommand{\NodeSubset}[1]{\mathbb{#1}}
\newcommand{\weight}[1]{\rho_{#1}}
\newcommand{\UnitHankelBlock}[1]{\mat{H}_{#1}}
\newcommand{\state}[1]{\vect{h}_{#1}}
\newcommand{\thetavar}{\vect{\theta}}
\newcommand{\betavar}{\vect{\beta}}
\newcommand{\Descendants}{\mathcal{D}}
\newcommand{\Descendantscompl}{\bar{\mathcal{D}}}
\newcommand{\Leaves}{\mathscr{L}}
\newcommand{\Xmat}[1]{\mat{X}_{#1}}
\newcommand{\Ymat}[1]{\mat{Y}_{#1}}
\newcommand{\Zmat}[1]{\mat{Z}_{#1}}
\newcommand{\Fmat}[1]{\mat{F}_{#1}}
\newcommand{\Gmat}[1]{\mat{G}_{#1}}

\newcommand{\zeroset}{\{ 0 \}}
\newcommand{\leftchild}{\text{\textsc{l}}}
\newcommand{\rightchild}{\text{\textsc{r}}}

\usepackage{tikz}
\usepackage{pgfplots}
\pgfplotsset{compat = newest}
\usepgfplotslibrary{fillbetween}

\begin{document}

\maketitle

\begin{abstract}
 We present a unification and generalization of what is known in the literature as sequentially and hierarchically semi-separable (SSS and HSS) representations for matrices. Describing rank-structured representations of (inverses of) sparse matrices whose adjacency graph is a tree, it is shown that these so-called tree quasi-separable (TQS) matrices inherit all the favorable algebraic properties of SSS and HSS under addition, products, and inversion. To arrive at these properties, we prove a key result that characterizes the conversion of any dense matrix into a TQS representation. Here, we specifically show through an explicit construction procedure that the generator sizes are dictated by the ranks of certain Hankel blocks of the matrix. Analogous to SSS and HSS, TQS matrices admit fast matrix-vector products and direct solvers provided the generator sizes are small. A sketch of the associated algorithms is provided.
\end{abstract}



\section{Introduction} 

Matrices in applied problems of interest often exhibit a structure of low rank in their off-diagonal blocks.  These structures have, for instance, been observed in the discretization of integral equations \cite{rokhlin1985rapid}, Schur complements of discretizations of PDEs \cite{chandrasekaran2010numerical,chandrasekaran2010numerical, xia2012robust}, certain Cauchy-like matrices \cite{chandrasekaran2008superfast, xia2012superfast}, evaluations of potentials \cite{greengard1987fast}, and companion matrices \cite{aurentz2015fast}, amongst others. Many frameworks have been proposed to efficiently represent these low-rank structures so that efficient linear algebra operations can be performed with such matrices. This includes the fast multiple method (FMM) \cite{greengard1987fast}, semi-separable and quasi-separable matrices \cite{vandebril2008matrix,bella2008computations, eidelman1999new}, sequentially semi-separable (SSS) matrices \cite{chandrasekaran2002fast,chandrasekaran2005some}, hierarchically semi-separable (HSS) matrices \cite{chandrasekaran2006fast,xia2010fast}, $\mathcal{H}$- and $\mathcal{H}^2$-matrices \cite{borm2003introduction,hackbusch2002data,hackbusch2015hierarchical}, and hierarchically off-diagonal low-rank (HODLR) matrices \cite{ambikasaran2013mathcal,aminfar2016fast}. These frameworks are all related and have their benefits, pitfalls, and special use cases.  A complete review of the subject goes beyond the scope of this paper.

This paper is concerned with the structure that arises from the inverse of a (block-)sparse matrix whose adjacency graph is a tree. This will generally be a dense matrix but with a special low-rank structure. In particular, if we look at sums and products of such matrices, they retain these structures, even though they no longer are the inverse of a sparse matrix. In this paper, we present a new class of representations for rank-structured matrices that actually can capture these structures exactly.  In fact, we show that these representations satisfy a graph-induced rank structure (GIRS) property if the corresponding graph of the associated graph-partitioned matrix is a tree (see \cite{chandrasekaran2019graph}). The representations are referred to \emph{tree quasi-separable} (TQS) matrices. 

Interestingly, TQS matrices simultaneously unify and generalize SSS and HSS matrices.  Apart from introducing a new family of rank-structured matrices, an important technical contribution of our paper is an algorithm that realizes a minimal TQS representation for any dense matrix. This algorithm is, in effect, a unification and generalization of the algorithms for doing the same for SSS and HSS matrices. The algorithm allows us to prove a result (\Cref{thm:universality}) that fully characterizes the properties of TQS representation for any dense matrix. Specifically, for a given tree structure and block partitioning of the matrix, the dimensions of the TQS generators are dictated by the ranks of certain matrix sub-blocks referred to as Hankel blocks. Through this fact, we show that TQS inherits the favorable algebraic properties of SSS and HSS matrices under addition, products, and inversion. Importantly, the inverse of a TQS matrix is again a TQS matrix of exactly the same rank profile if the input and output dimensions of the partitions are chosen equally. TQS matrices allow for fast inversion algorithms and matrix-vector products provided the generator sizes are small. We present one such fast direct solver using the sparse embedding trick introduced in \cite{chandrasekaran2007fast} for HSS matrices. 

It is well-known that SSS and HSS matrices are used extensively in practice. TQS matrices are a more flexible drop-in replacement for HSS and SSS matrices. We anticipate that this flexibility can lead to even further improvements in many applications, particularly for networked dynamical systems on graphs \cite{verhaegen2022data}. Although applications are the subject of future work, we present some illustrative examples of sparse matrices where the choice of a TQS representation is both natural and efficient. 

The remainder of this paper is outlined as follows. In \Cref{sec:review}, we give a quick review of SSS and HSS matrices. \Cref{sec:TQSdef} introduces TQS matrices in detail. The algebraic properties of TQS matrices are discussed in \Cref{sec:TQSalgebraic}. In \Cref{sec:construction}, we discuss the construction algorithm, and along with it the proof of \Cref{thm:universality}.  \Cref{sec:TQSsystems} covers the fast matrix-vector multiply algorithm and the procedure for efficiently solving a linear system involving TQS matrices.


\section{A brief recap on SSS and HSS matrices} \label{sec:review}
This section briefly reviews the definitions of SSS and HSS matrices.
We emphasize that we use a rather unconventional, and admittedly, redundant notation in the definitions of both SSS and HSS matrices. However, this is done intentionally to establish a direct link with TQS matrices that will require more detailed notation. \Cref{sec:SSS} covers SSS matrices and \Cref{sec:HSS} covers HSS matrices. \Cref{sec:HSSandSSSalgebra} briefly summarizes performing algebra with SSS and HSS matrices. For more details on the algorithms themselves, we recommend \cite{chandrasekaran2005some,chandrasekaran2002fast}
for SSS and \cite{xia2010fast, chandrasekaran2006fast, wang2013efficient, xia2012complexity} for HSS. A more tutorial-styled introduction to the subject is given in \cite{shivnithinabhe}.

\subsection{SSS matrices} \label{sec:SSS}
Sequentially semi-separable (SSS) matrices were first introduced in \cite{dewilde1998time} in a study to generalize systems theory to the time-varying case.  To define SSS matrices, let $m_i,n_i\in \naturals \cup \zeroset$ for $i\in \left\lbrace 1,\ldots,K\right\rbrace$ and $\weight{(i,i+1)}\in \naturals \zeroset$ for $i\in \left\lbrace 1,\ldots,K-1\right\rbrace$ and introduce the matrices $\Dmat^i  \in \field^{m_i \times n_i}$ for $i\in \left\lbrace 1,\ldots,K\right\rbrace$,  $\Bmat^i_{i+1} \in \field^{\weight{(i,i+1)} \times n_i }$ for $i\in \left\lbrace 1,\ldots,K-1\right\rbrace$,  
$\Umat^i_{i+1,i-1} \in \field^{\weight{(i,i+1)} \times \weight{(i-1,i)}}$ for $i\in \left\lbrace 2,\ldots,K-1\right\rbrace$, $\Cmat^i_{i-1} \in \field^{m_i \times \weight{(i-1,i)}}$ for $i\in \left\lbrace 2,\ldots,K\right\rbrace$, $\Pmat^i_{i-1} \in \field^{\weight{(i,i-1)} \times n_i}$ for $i\in \left\lbrace 2,\ldots,K\right\rbrace$, $\Wmat^i_{i-1,i+1 } \in \field^{\weight{(i,i-1)} \times \weight{(i+1,i)}}$ for $i\in \left\lbrace 2,\ldots,K-1\right\rbrace$, $\Qmat^i_{i+1 } \in \field^{m_i\times \weight{(i+1,i)}}$ for $i\in \left\lbrace 1,\ldots,K\right\rbrace$. Next, define    
\begin{equation}
    \mat{A}_{ij } := \begin{cases} 
    \Dmat^{i} & i=j  \\
    \Cmat^{i}_{i-1} \Umat^{i-1}_{i,i-2}  \cdots \Umat^{j+1}_{j+2,j} \Bmat^{j}_{j+1}   & i> j \\
    \Qmat^{i}_{i+1} \Wmat^{i+1}_{i,i+2}  \cdots \Wmat^{j-1}_{j-2,j} \Pmat^{j}_{j-1}   & i < j
    \end{cases}. \label{eq:SSS}
\end{equation}
A sequentially semi-separable (SSS) matrix $\mat{A}\in\field^{M \times N}$, with $M=\sum^K_{i=1} m_i$ and $N=\sum^K_{i=1} n_i$, is the block-partitioned matrix with block entries specified by \eqref{eq:SSS}. For example, in the case of $K=4$, an SSS matrix takes on the form
\begin{equation}
\mat{A} = \begin{bmatrix}
\Dmat^1               &          \Qmat^1_2 \Pmat^2_1             &    \Qmat^1_2 \Wmat^2_{1,3} \Pmat^3_2  &    \Qmat^1_2 \Wmat^2_{1,3} \Wmat^3_{2,4} \Pmat^4_3    \\
\Cmat^{2}_1 \Bmat^1_2   &  \Dmat^2  &  \Qmat^2_3 \Pmat^3_{2} & \Qmat^2_3 \Wmat^3_{2,4} \Pmat^4_3  \\
\Cmat^{3}_2 \Umat^{2}_{1,3} \Bmat^1_2          &  \Cmat^{3}_2 \Bmat^2_3         &  \Dmat^3 & \Qmat^3_4 \Pmat^4_3      \\
\Cmat^{4}_2 \Umat^{3}_{4,2} \Umat^{2}_{3,1} \Bmat^1_2     &  \Cmat^{4}_2 \Umat^{3}_{4,2} \Bmat^2_3  & \Cmat^{4}_3 \Bmat^3_2   & \Dmat^4
\end{bmatrix}. 
\label{eq:SSSSn4}
\end{equation}
The entries of the SSS representation \eqref{eq:SSS} can be seen as the result of decomposing $\mat{A}$ as the sum of a causal and anti-causal linear-time-variant (LTV) system on a line graph. At this point of our discussion, it may be important to remark that the phrase ``sequentially semi-separable'' has been chosen rather inconveniently in the literature since SSS matrices are effectively a representation for the much richer class of quasi-separable matrices \cite{bella2008computations,eidelman1999new}. The prefix ``sequentially quasi-separable'' would have been more appropriate. 

\subsection{HSS matrices} \label{sec:HSS}

While SSS matrices have their origins in systems theory, hierarchically semi-separable (HSS) matrices arose independently to simplify the algebra of the fast multipole method (FMM) so that it has favorable properties under inversion. To define HSS matrices, we recursively partition a matrix $\mat{A}\in\field^{M \times N}$ in a hierarchic manner. This process is best described through a binary tree. 

Assume that the nodes of the binary tree are indexed by a \emph{post-ordered} traversal of the tree and let $r$ denote its root node. Furthermore, let $\leftchild(l)$ (respectively, $\rightchild(l)$) denote the left (respectively, right) child of node $l$ in the binary tree. If $l$ is a terminating point of the recursion, or equivalently, a leaf node of the binary tree, we set $\mat{A}^{l} = \Dmat^{l}\in\field^{m_{l} \times n_{l}}$. On the other hand, if $l$ is a non-terminating point of the recursion, we set 
\begin{equation}
    \mat{A}^{l} = \begin{bmatrix}
        \mat{A}^{\leftchild(l)}   & \mathcal{Q}^{\leftchild(l)}_{ l }   \Vmat^{l}_{ \leftchild(l) , \rightchild(l) }  \mathcal{B}^{\rightchild(l) }_{l}  \\
\mathcal{Q}^{\rightchild(l) }_{l}   \Vmat^{l}_{\rightchild(l),\leftchild(l)}  \mathcal{B}^{\leftchild(l)}_{l}  & \mat{A}^{\rightchild(l)}
    \end{bmatrix},
\end{equation}
where $\mat{A}^{\leftchild(l)}\in\field^{ m_{\leftchild(l)} \times   n_{\leftchild(l)} }$, $\mat{A}^{\rightchild(l)}\in\field^{m_{\rightchild(l)} \times  n_{\rightchild(l)}}$,  
$m_{l} = m_{\leftchild(l)} + m_{\rightchild(l)}$, and $n_{l} = n_{\leftchild(l)} + n_{\rightchild(l)}$. Furthermore, $\mathcal{Q}^{ \leftchild(l)}_{l} = \Qmat^{ \leftchild(l) }_{l} \in \field^{m_{\leftchild(l)} \times \weight{(l, \leftchild(l))}}$ (respectively, $\mathcal{Q}^{ \rightchild(l)}_{l} = \Qmat^{ \rightchild(l) }_{l} \in \field^{m_{\rightchild(l)} \times \weight{(l, \rightchild(l))}}$) if $\leftchild(l)$ (respectively, $\rightchild(l)$) is a leaf node of the binary tree. Otherwise,
$$
\mathcal{Q}^{\leftchild(l)}_{l} = \begin{bmatrix}
  \mathcal{Q}^{\leftchild(\leftchild(l))}_{ \leftchild(l) }  \Wmat^{\leftchild(l)}_{ \leftchild(\leftchild(l))  , l }     \\
  \mathcal{Q}^{\rightchild(\leftchild(l))}_{\leftchild(l)}   \Wmat^{\leftchild(l)}_{ \rightchild(\leftchild(l)),l}    
\end{bmatrix}, \qquad \mathcal{Q}^{\rightchild(l)}_{l} = \begin{bmatrix}
  \mathcal{Q}^{\leftchild(\rightchild(l))}_{\rightchild(l) }  \Wmat^{\rightchild(l)}_{\leftchild(\rightchild(l)),l}     \\
  \mathcal{Q}^{\rightchild(\rightchild(l))}_{\rightchild(l)}   \Wmat^{\rightchild(l)}_{ \rightchild(\rightchild(l)),l} 
\end{bmatrix}.
$$
with $\Wmat^{\leftchild(l)}_{ \leftchild(\leftchild(l)),l} \in \field^{ \weight{(\leftchild(l), \leftchild(\leftchild(l)))} \times \weight{(l, \leftchild(l))}}$, $\Wmat^{\leftchild(l)}_{ \rightchild(\leftchild(l)),l} \in \field^{ \weight{(\leftchild(l), \rightchild(\leftchild(l)))} \times \weight{(l, \leftchild(l))}}$, \\
$\Wmat^{\rightchild(l)}_{ \rightchild(\leftchild(l)),l} \in \field^{ \weight{(\rightchild(l), \rightchild(\leftchild(l)))} \times \weight{(l, \rightchild(l))}}$, $\Wmat^{\rightchild(l)}_{ \rightchild(\rightchild(l)),l} \in \field^{ \weight{(\rightchild(l), \rightchild(\rightchild(l)))} \times \weight{(l, \rightchild(l))}}$. Similarly, $\mathcal{B}^{ \leftchild(l)}_{l} = \Bmat^{ \leftchild(l) }_{l} \in \field^{ \weight{(\leftchild(l),l)}  \times n_{\leftchild(l)}}$ (respectively, $\mathcal{B}^{ \rightchild(l)}_{l} = \Bmat^{ \rightchild(l) }_{l} \in \field^{ \weight{(\rightchild(l), l)} \times n_{\rightchild(l)} }$) if $\leftchild(l)$ (respectively, $\rightchild(l)$) is a leaf node of the binary tree. Otherwise,
\begin{eqnarray*}
 \mathcal{B}^{ \leftchild(l)}_{l}    & = & \begin{bmatrix}
   \Umat^{\leftchild(l)}_{l , \leftchild(\leftchild(l)) } \mathcal{B}^{\leftchild(\leftchild(l))}_{\leftchild(l)}   &  \Umat^{ \leftchild(l)   }_{ l , \rightchild(\leftchild(l)) } \mathcal{B}^{\rightchild(\leftchild(l))}_{\leftchild(l)} 
 \end{bmatrix},  \\
  \mathcal{B}^{ \rightchild(l)}_{l}  & = &   \begin{bmatrix}
   \Umat^{\rightchild(l)}_{l,\leftchild(\rightchild(l))} \mathcal{B}^{\leftchild(\rightchild(l))}_{\rightchild(l)}    & \Umat^{\rightchild(l)}_{ l , \rightchild(\rightchild(l)) } \mathcal{B}^{\rightchild(\rightchild(l))}_{\rightchild(l)}
 \end{bmatrix} 
\end{eqnarray*}
with $\Umat^{\leftchild(l)}_{l, \leftchild(\leftchild(l))} \in \field^{ \weight{(\leftchild(l),l)}  \times \weight{( \leftchild(\leftchild(l)),\leftchild(l))} }$, $\Umat^{\leftchild(l)}_{ \rightchild(\leftchild(l)),l} \in \field^{ \weight{(\leftchild(l), l)}  \times  \weight{(\rightchild(\leftchild(l)), \leftchild(l))}  }$,  \\ $\Umat^{\rightchild(l)}_{ \leftchild(\rightchild(l)),l} \in \field^{ \weight{(\rightchild(l), l)}   \times \weight{(\leftchild(\rightchild(l)), \rightchild(l))} }$,  $\Umat^{\rightchild(l)}_{ l, \rightchild(\rightchild(l))} \in \field^{ \weight{(\rightchild(l),l)}  \times \weight{(\rightchild(\rightchild(l)) , \rightchild(l))}  }$. At the root level, we set $\mat{A} = \mat{A}^{r}\in\field^{m_{r} \times n_{r}}$ with $m_{r}=:M$ and $n_{r}=:N$. For example, the matrix 
\begin{eqnarray}
     \mat{A} & = &  \begin{bmatrix} \Dmat^{1} &  
\Qmat^{1}_{3} \Vmat^{3}_{1,2} \Bmat^{2}_{3}  &   \Qmat^1_{5}  \Wmat^3_{1,5} \Vmat^{5}_{3,4} \Bmat^{4}_{5} \\
 \Qmat^{1}_{3} \Vmat^{3}_{2,1} \Bmat^{1}_{3}   & \Dmat^{2}  &      \Qmat^2_{5}  \Wmat^3_{1,5} \Vmat^{5}_{3,4} \Bmat^{4}_{5}           \\
    \Qmat^4_{5}  \Vmat^5_{4,3} \Umat^{3}_{5,1} \Bmat^{1}_{3}     &    \Qmat^4_{5}  \Vmat^5_{4,3} \Umat^{3}_{5,2} \Bmat^{2}_{3}          &  \Dmat^{4} 
\end{bmatrix} \label{eq:HSSexample} \\
   & = & \begin{bmatrix}
        \mat{A}^{3}   & \mathcal{Q}^{3}_{5}   \Vmat^{5}_{3,4}  \mathcal{B}^{4 }_{5}  \\
\mathcal{Q}^{4}_{5}  \Vmat^{5}_{4,3}  \mathcal{B}^{3}_{5}  & \mat{A}^{4} \end{bmatrix} \nonumber \\
   & = &  \mat{A}^5  \nonumber
\end{eqnarray}
 is an HSS matrix.
\subsection{Algebra with SSS and HSS matrices} \label{sec:HSSandSSSalgebra}
SSS and HSS matrices share common algebraic properties. Any dense matrix can be converted into an SSS or HSS matrix. The dimensions of the generators, and thus the efficiency of the representation, are specified by the ranks of the off-diagonal (i.e., so-called Hankel) blocks. Sums and products of SSS (respectively HSS) matrices are again SSS (respectively HSS) but with a doubling in the size of the generators. The inverse of an SSS (respectively HSS) is an SSS (respectively HSS) matrix of the same generator dimensions. SSS (respectively HSS) representations of matrices that have small Hankel block ranks can be multiplied with vectors in linear time. The same holds for the solve operation. For the latter,  one common approach to both representations is a lifting technique that solves the linear system as a larger block-sparse system \cite{chandrasekaran2007fast}. All these commonalities do not come from nowhere, since SSS and HSS belong to the same family of a more general class of matrices. 

\section{Tree quasi-separable (TQS) matrices} \label{sec:TQSdef}
This section formally introduces TQS matrices. To do so, we first introduce some terminology in \Cref{sec:treegraphs} and \Cref{sec:graphpartitioned}. The actual definition of TQS matrices is given in \Cref{sec:TQSdefinition}. Finally, \Cref{sec:HSSandSSSandTQS} describes SSS and HSS as special cases of TQS matrices.

\subsection{Tree graphs} \label{sec:treegraphs}
Let $\Graph = (\NodeSet,\EdgeSet)$, with node set $\mathbb{V}=\lbrace1,2,\ldots,K\rbrace$, be a \emph{connected acyclic undirected graph}. Here, for reasons that will be clear later, the edge set $\EdgeSet$ is \emph{unconventionally} interpreted as a collection \emph{ordered pairs} of nodes (instead of unordered pairs!), but with the property that $(i,j)\in \EdgeSet$ if and only if $(j,i)\in \EdgeSet$, i.e., the edges in both directions are included.  The number of edges incident to a node $i \in \NodeSet$, i.e. its degree, is denoted by $\degnode(i)$. Since $\Graph$ is acyclic, every pair of nodes $i,j \in \Graph$ is connected by one, and only one, \emph{path}\footnote{Here we do not allow for self-intersecting paths. A ``path'' passing from node $i$ to node $j$ by passing through $i$ and going to $k$, to subsequently come back to node $j$, is \emph{not} considered a valid path!} 
$$\Path(i,j)=i-w_2-\cdots-w_{p-1}-j$$
of specific length $p$. 

The graph $\Graph$ may be interpreted as a rooted tree\footnote{Although one may work without such an interpretation, this viewpoint shall be useful for deriving some of the results in the paper. Specifically, it allows the labeling of the generating matrices, which in turn simplifies validating the correctness of \Cref{alg:construction}.}. 
Indeed, if $\kNeighbor(i;k)$ denotes the \emph{set of $k$-neighbors} of node $i\in\mathbb{V}$, i.e., the set of nodes that can be reached by $i$ through traversing a path of length $k$, then a tree $\Graph(r)$ is induced by picking any $r\in \NodeSet$ and setting 
\begin{displaymath}
\NodeSet_0 = \lbrace r \rbrace, \quad \NodeSet_1 = \kNeighbor(r;1),\quad \NodeSet_2 = \kNeighbor(r;2),\quad\ldots,\quad \NodeSet_L = \kNeighbor(r;L).  
\end{displaymath}
The depth of the tree specified by the partition $\NodeSet = \NodeSet_0 \cup \NodeSet_1 \cup \ldots \cup \NodeSet_{L}$ equals $L$ and is defined as the smallest number with the property $\kNeighbor(r; L+1) = \emptyset$.  The level $l$ at which a node $i\in\NodeSet_l$ resides within the tree is denoted by $\Level(i)$ (note that this depends implicitly on the choice of root node made earlier!). An edge $(i,j)\in\EdgeSet$ with $\Level(i)=\Level(j)-1$ is called an \emph{up-edge} since it is directed towards the root node. Likewise, if $\Level(i)=\Level(j)+1$ it is called a \emph{down-edge} since it is directed away from the root node.

Within the setting of the tree $\Graph(r)$, a node $j\in \NodeSet$ is consider to be child of $i\in \kNeighbor(j) = \kNeighbor(j;1)$ if $i\in\NodeSet_{l}$ and $j\in\NodeSet_{l+1}$. Vice versa, $i$ is the parent of node $j$. The notion of children and parents may be further generalized: $j\in \NodeSet$ is a $k$-child (with $k>0$) of $i\in\kNeighbor(j;k)$ (and vice versa $i$ is the $k$-parent of node $j$) if $i\in\NodeSet_{l}$ and $j\in\NodeSet_{l+k}$. Two nodes $i,j\in\NodeSet_l$ are siblings ($k$-siblings) if they both share the same parent node $w\in \NodeSet_{l-1}$ ($w\in \NodeSet_{l-k}$).  Naturally, a node will have at most one $k$-parent, but it may have many $k$-children or $k$-siblings.  The parent ($k$-parent) of a node $i\in \NodeSet$ is denoted by $\Parent(i)$ ($\Parent(i;k)$) and this set may be empty (e.g., for the root node). The set of children ($k$-children) of node $i\in \NodeSet$ is denoted by $\Children(i)$ ($\Children(i;k)$). The set of siblings ($k$-siblings) of node $i\in \NodeSet$ is denoted by $\Siblings(i)$ ($\Siblings(i;k)$). The set of descendants of node $i\in \NodeSet$ and including $i$ itself is denoted by
\begin{displaymath}
    \Descendants(i):= i \cup \Children(i;1) \cup \Children(i;2) \cup \ldots \cup \Children(i;L-\Level(i)).
\end{displaymath}
and we set  $\Descendantscompl(i):= \NodeSet \setminus \Descendants(i)$. A node $i\in\mathbb{V}$ is called a \emph{leaf node} if its set of descendants is empty. Note that leaf nodes may reside at any level of the tree. The set of leaf nodes of an acyclic fully connected graph is denoted by $\Leaves(\Graph)$. 

\subsection{Graph-partitioned matrices} \label{sec:graphpartitioned}
We can associate the graph $\Graph=(\NodeSet,\EdgeSet)$ with a block-partitioned matrix. To do this, associate each node with an input of size $n_i\in \naturals \cup \zeroset$ and an output of size $m_i \in \naturals \cup \zeroset$. Note that we allow the sizes of the inputs and outputs to be \emph{empty}! Now, let 
\begin{displaymath}
    M=\sum_{i\in \NodeSet} m_i, \quad N=\sum_{i\in \NodeSet} n_i.
\end{displaymath}
We may then introduce a matrix $\TQS\in \field^{M\times N}$ that is assembled from the system of linear equations
\begin{equation}\label{eq:block_matrix_multiplication}
 \vect{b}_{i} = \sum_{j\in \NodeSet} \TQS\nodeid{i,j} \vect{x}_j, \quad i\in\NodeSet.  
\end{equation}
That is,
\begin{equation}
    \TQS := \begin{blockarray}{ccccc}
           & 1 & 2 & \cdots & K \\
        \begin{block}{c[cccc]}
     1 &  \TQS\nodeid{1,1} & \TQS\nodeid{1,2} & \cdots &\TQS\nodeid{1,K} \\
    2 & \TQS\nodeid{2,1} & \TQS\nodeid{2,2} & \cdots &\TQS\nodeid{2,K} \\
    \vdots & \vdots & \vdots &  & \vdots \\
    K & \TQS\nodeid{K,1} & \TQS\nodeid{K,2} & \cdots &\TQS\nodeid{K,K} \\
        \end{block}
    \end{blockarray},
\end{equation}
where $\TQS\nodeid{i,j}\in \field^{m_i\times n_j}$ is the matrix associated with the contribution of the input and node $i\in\NodeSet$ to the output at node $j\in\NodeSet$. Given two subsets $\NodeSubset{A} = \lbrace i_1,  \ldots, i_A \rbrace\subset \NodeSet$ and $\NodeSubset{B} = \lbrace j_1, \ldots, j_B \rbrace\subset \NodeSet$ we sub-matrices of $\TQS$ with the notation
\begin{displaymath}
    \TQS \nodeid{\NodeSubset{A},\NodeSubset{B}} = \begin{blockarray}{cccc} 
     & j_{1} & \cdots & j_{B} \\
     \begin{block}{c[ccc]}
        i_1 &   \TQS\nodeid{i_1,j_1} &  \cdots &\TQS\nodeid{i_1,j_B} \\
    \vdots  & \vdots &   & \vdots \\
    i_A & \TQS\nodeid{i_A,j_1} &  \cdots &\TQS\nodeid{i_A,j_B} \\
    \end{block}
    \end{blockarray}.
\end{displaymath}
Actual entries of the matrix $\TQS$ or its block-components $\TQS\nodeid{i,j}$ are accessed using a square-bracket notation, i.e., $\TQS[k,l]$ and $\TQS\nodeid{i,j}[k,l]$ refer to the $(k,l)$-th entry of $\TQS$ and $\TQS\nodeid{i,j}$, respectively.

The matrix $\TQS$ alongside with $\Graph$ is referred to as a \emph{graph-partitioned matrix}. Note that this definition is agnostic to whether $\Graph$ is a tree or not. Nonetheless, in this paper, we will limit ourselves to only acyclic undirected graphs that may also be viewed as trees.

\subsection{Definition of TQS matrices} \label{sec:TQSdefinition}
A tree quasi-separable (TQS) matrix is a specific construction of a graph-partitioned matrix $\TQS$ associated with a fully connected and acyclic graph $\Graph = (\NodeSet,\EdgeSet)$. To form a TQS matrix from $\Graph$, we equip the graph with weights on the edges of the graph: every edge $(i,j) \in \mathbb{E}$ is associated with a weight $\weight{(i,j)} \in \naturals \cup \zeroset$. Furthermore, every edge $(i,j) \in \mathbb{E}$ is associated with a state vector $\state{(i,j)}\in \field^{\weight{(i,j)}}$. Specifically, $\state{(i,j)}$ describes a state computed in node $i$ and transmitted to node $j$.  The set  $\lbrace \weight{e} \rbrace_{e\in \EdgeSet}$ is referred to as the \emph{rank-profile} on $\Graph$. It collectively describes the state dimensions of all the edges. The TQS matrix is formed by running an input-driven LTV dynamical system on the graph.  To describe this LTV system, the following matrices are introduced (see also \Cref{fig:TQSelements}):
\begin{itemize}
\item an \emph{input-to-output} operator $\Dmat^k \in \field^{m_k\times n_k}$ for every node $k\in\NodeSet$ describing a mapping from input $\vect{x}_k$ to output $\vect{b}_k$,
\item an \emph{input-to-edge} operator $\Inp^k_j \in \field^{\weight{(k,j)}\times n_k}$ for every node $k\in \kNeighbor(j)$ and adjoining edge $(k,j)\in \EdgeSet$ describing a mapping from input $\vect{x}_k$ to state $\state{(k,j)}$, 
\item an \emph{edge-to-edge} operator $\Trans^k_{i,j}\in \field^{\weight{(j,k)}\times \weight{(j,k)}}$ for every pair of adjoining edges $(j,k)\in\EdgeSet$ and $(k,i)\in\EdgeSet$ (with $i,j\in \kNeighbor(k)$) describing a mapping from state $\state{(j,k)}$ to state $\state{(k,i)}$,
\item and an \emph{edge-to-output} operator $\Out^k_i \in \field^{m_k \times\weight{(i,k)}}$ for every edge $(i,k)\in\EdgeSet$ and adjoining node $i\in\kNeighbor(k)$ describing a mapping from state $\state{(j,k)}$ to output $\vect{b}_k$.
\end{itemize}
 Notice that matrices (operators) belong to a node indicated by the superscript, while the subscripts implicitly refer to edges attached to the respective node, indicating the destination of the data being computed in that node. A TQS matrix is then defined as follows.
\begin{definition}[tree quasi-separable matrices]\label{def:TQS}
    Let $\Graph = (\NodeSet,\EdgeSet)$ be a connected acyclic undirected graph. A tree quasi-separable (TQS) matrix with input dimensions $\lbrace n_i \rbrace_{i \in \NodeSet}$, output dimensions $\lbrace m_i \rbrace_{i \in \NodeSet}$, and rank-profile $\lbrace \weight{e} \rbrace_{e\in \EdgeSet}$, is a graph-partitioned matrix $\TQS\in\field^{M\times N}$ of size $M=\sum_{i\in \NodeSet} m_i$ by $N=\sum_{i\in \NodeSet} n_i$  whose block-entries are specified by
    \begin{equation}
\TQS\nodeid{i,j} = \begin{cases}
    \Dmat^i & i=j \\
      \Out^i_{p_{\nu-1}} \Trans^{p_{\nu-1}}_{i,p_{\nu-2}}      \cdots \Trans^{p_2}_{p_3,p_1} \Trans^{p_1}_{p_2,j} \Inp^j_{p_1} & i \ne j,
\end{cases}
    \end{equation} 
    with $j=p_0-p_1-\cdots-p_{\nu}=i$ being the unique path from node $j\in\NodeSet$ to node $i\in\NodeSet$.
\end{definition}

\begin{figure}[h]
  \centering
  \begin{subfigure}{0.23\textwidth}
    \centering
    \begin{tikzpicture}
      \node[circle, draw, fill=black, inner sep=1.5pt] (myNode) at (0,0) {};
  \node[right] at (myNode.north) {$k$}; 
    \end{tikzpicture}
    \caption{Every node $k\in \NodeSet$ is equipped with an \emph{input-to-output} operator.}
  \end{subfigure}
  \hfill
  \begin{subfigure}{0.23\textwidth}
    \centering
    \begin{tikzpicture}
        \node[circle, draw, fill=black, inner sep=1.5pt] (myNode) at (0,0) {};
  \node[right] at (myNode.north) {$k$}; 
  \draw[-] (myNode) -- ++(330:1.2cm) node[midway, right, pos=1.0] {$(k,j)$};
    \end{tikzpicture}
    \caption{Every node $k\in \kNeighbor(j)$ with adjoining edge $(k,j)\in \EdgeSet$ is equipped with an \emph{input-to-edge} operator.}
  \end{subfigure}
 \hfill
  \begin{subfigure}{0.23\textwidth}
    \centering
    \begin{tikzpicture}
     \node[circle, draw, fill=black, inner sep=1.5pt] (myNode) at (0,0.5) {};
  \node[right] at (myNode.north) {$k$}; 
  \draw[-] (myNode) -- ++(290:1.2cm) node[midway, right, pos=1.0] {$(k,j)$};
    \draw[-] (myNode) -- ++(250:1.2cm) node[midway, left, pos=1.0] {$(j,k)$};
    \end{tikzpicture}
    \caption{Every pair of adjoining edges $(j,k)\in\EdgeSet$ and $(k,i)\in\EdgeSet$ is equipped with an \emph{edge-to-edge} operator.}
  \end{subfigure}
  \hfill
  \begin{subfigure}{0.23\textwidth}
    \centering
    \begin{tikzpicture}
     \node[circle, draw, fill=black, inner sep=1.5pt] (myNode) at (0,0) {};
  \node[right] at (myNode.north) {$k$}; 
  \draw[-] (myNode) -- ++(210:1.2cm) node[midway, left, pos=1.0] {$(i,k)$};
    \end{tikzpicture}
    \caption{Every edge $(i,k)\in\EdgeSet$ with adjoining node $i\in\kNeighbor(k)$ is equiped with an \emph{edge-to-output} operator.}
  \end{subfigure}
  \caption{TQS operators associated with elements of the graph $\Graph$.} \label{fig:TQSelements}
\end{figure}

The notation of the generating matrices of the TQS representation in definition~\ref{def:TQS} is \emph{agnostic} to a tree order. Later on, in \Cref{sec:construction}, we shall address the realization problem \emph{or} the problem of constructing a TQS representation from a dense graph-partitioned matrix. To aid this construction, it shall be useful to ``color'' the generating matrices with respect to their orientation in the graph. Particularly, once a root node $r\in\NodeSet$ is chosen for the graph, we introduce additional \emph{nomenclature} to distinguish between different input-to-edge, edge-to-edge, and edge-to-output operators\footnote{Mind that these operators are well-defined irrespective of the choice of the root.}:
\begin{itemize}
    \item An \emph{input-to-edge} operator $\Inp^k_j$ is denoted by $\Bmat^k_j$ if $j\in\NodeSet$ is a {parent} of $k\in\NodeSet$  (i.e., $j=\Parent(k)$), and by $\Pmat^k_j$ if $j\in\NodeSet$ is a {child} of $k\in\NodeSet$  (i.e., $j\in\Children(k)$).
    \item  An \emph{edge-to-edge} operator $\Trans^k_{i,j}$ is denoted by $\Umat^k_{i,j}$ if $j\in\NodeSet$ and $i\in\NodeSet$ are respectively a child and parent of $k\in\NodeSet$ (i.e., $j\in\Children(k)$ and $i=\Parent(k)$), by $\Vmat^k_{i,j}$ if $i,j\in\NodeSet$ are siblings of $k\in\NodeSet$ (i.e. $i,j\in\Siblings(k)$), and by $\Wmat^k_{i,j}$ if $j\in\NodeSet$ and $i\in\NodeSet$ are respectively a parent and child of $k\in\NodeSet$ (i.e., $j=\Parent(k)$ and $i\in\Children(k)$).
    \item An \emph{edge-to-output} operator $\Out^k_i$ is denoted by $\Cmat^k_i$ if $i\in\NodeSet$ is a \emph{child} of $k\in\NodeSet$ (i.e., $i\in\Children(k)$), and by $\Qmat^k_i$ if $i\in\NodeSet$ is a \emph{parent} of $k\in\NodeSet$ (i.e., $i=\Parent(k)$).
\end{itemize}
The input-to-output, input-to-edge, edge-to-edge, and edge-to-output operators can systematically be organized in so-called \emph{spinner matrices} or transition maps at every node $k\in\NodeSet$. Let $i_1,i_2,\ldots,i_p\in\Children(k)$ be an enumeration of the set of children and $j=\Parent(k)$  the parent node (see \Cref{fig:spinnerillustration}). At every node $k\in\NodeSet$, the following relations must be satisfied
\begin{equation}
      \begin{bmatrix} \state{(k,i_1)} \\ \state{(k,i_2)} \\  \vdots  \\ \state{(k,i_p)} \\  \state{(k,j)} \\ \vect{b}_k \end{bmatrix} = \begin{bmatrix}
 0 & \Vmat^{k}_{i_1,i_2}  & \cdots & \Vmat^{k}_{i_1,i_p} &  \Wmat^{k}_{i_1,j}  & \Pmat^{k}_{i_1}   \\
 \Vmat^{k}_{i_2,i_1}    & 0 & & \Vmat^{k}_{i_2,i_p}  & \Wmat^{k}_{i_2,i_1}    &  \Pmat^{k}_{i_2}  \\
\vdots &  \vdots    &  & \ddots & \vdots & \vdots      \\
 \Vmat^{k}_{i_p,i_1}    & \Vmat^{k}_{i_p,i_2}  & \cdots & 0 & \Wmat^{k}_{i_p,j}    &  \Pmat^{k}_{i_p}  \\
 \Umat^{k}_{j,i_i}    & \Umat^{k}_{j,i_2} & \cdots & \Umat^{k}_{j,i_p} & 0 & \Bmat^k_j   \\
 \Cmat^{k}_{i_1} & \Cmat^{k}_{i_2} & \cdots & \Cmat^{k}_{i_p} &  \Qmat^{k}_{j} & \Dmat^k  
\end{bmatrix}   \begin{bmatrix} \state{(i_1,k)} \\ \state{(i_2,k)} \\  \vdots  \\ \state{(i_p,k)} \\  \state{(j,k)} \\ \vect{x}_k \end{bmatrix}  \label{eq:transitionmaps}
\end{equation}
The spinner matrices of the TQS representation are
\begin{equation}
      \Spinner^{k} := \begin{blockarray}{ccccccc}
& i_1 & i_2& \cdots & i_p & j & \Inp   \\
\begin{block}{c[cccccc]} 
i_1 & 0 & \Vmat^{k}_{i_1,i_2}  & \cdots & \Vmat^{k}_{i_1,i_p} &  \Wmat^{k}_{i_1,j}  & \Pmat^{k}_{i_1}   \\
i_2 & \Vmat^{k}_{i_2,i_1}    & 0 & & \Vmat^{k}_{i_2,i_p}  & \Wmat^{k}_{i_2,i_1}    &  \Pmat^{k}_{i_2}  \\
\vdots &  \vdots    &  & \ddots & \vdots & \vdots    &  \vdots  \\
i_p & \Vmat^{k}_{i_p,i_1}    & \Vmat^{k}_{i_p,i_2}  & \cdots & 0 & \Wmat^{k}_{i_p,j}    &  \Pmat^{k}_{i_p}  \\
j & \Umat^{k}_{j,i_i}    & \Umat^{k}_{j,i_2} & \cdots & \Umat^{k}_{j,i_p} & 0 & \Bmat^k_j   \\
\Out & \Cmat^{k}_{i_1} & \Cmat^{k}_{i_2} & \cdots & \Cmat^{k}_{i_p} &  \Qmat^{k}_{j} & \Dmat^k  \\
\end{block}
\end{blockarray}, \qquad k\in\NodeSet.
\end{equation}
The spinner matrices reveal how many numerical entries are involved in a TQS matrix. Specifically, this number equals
\begin{displaymath}
   \sum_{i\in\NodeSet} \left( m_i + \sum_{j \in \kNeighbor(i)} r_{(i,j)}  \right)  \left( n_i + \sum_{j \in \kNeighbor(i)} r_{(j,i)}  \right)  -  \sum_{j \in \kNeighbor(i)} r_{(i,j)} r_{(j,i)}. 
\end{displaymath}
To illustrate our definition with an example, consider the following tree
    
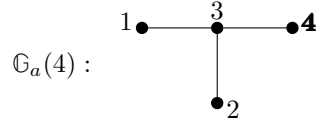
\begin{equation*}
    \begin{tikzpicture}
     \node at (-1.2,0) {$\Graph_a(4):$};
     \node[circle, draw, fill=black, inner sep=1.5pt] (myNodeA) at (2,0.5) {};
  \node[right] at (myNodeA.north) {$\pmb{\bm{4}}$};
  \node[circle, draw, fill=black, inner sep=1.5pt] (myNodeB) at (1,0.5) {};
  \node[above] at (myNodeB.center) {$3$};
   \node[circle, draw, fill=black, inner sep=1.5pt] (myNodeC) at (0,0.5) {};
  \node[left] at (myNodeC.north) {$1$}; 
    \node[circle, draw, fill=black, inner sep=1.5pt] (myNodeD) at (1,-0.5) {};
  \node[right] at (myNodeD.south) {$2$}; 
    \draw (myNodeA) -- (myNodeB);
    \draw (myNodeB) -- (myNodeC);
    \draw (myNodeB) -- (myNodeD);
\end{tikzpicture} 
\end{equation*}
with node $4$ (highlighted in bold) designated to be the root node. The TQS matrix for the corresponding tree $\Graph_a(4)$
has the structure
\begin{displaymath}
\TQS = \begin{blockarray}{ccccc}
& 1 & 2  & 3 & 4  \\
\begin{block}{c[cccc]} 
1 & \Dmat^1 &   \Qmat^1_3 \Vmat^3_{1,2} \Bmat^2_3   &  \Qmat^1_3 \Pmat^3_1  & \Qmat^1_3 \Wmat^3_{1,4} \Pmat^4_3  \\
 2 & \Qmat^2_3 \Vmat^3_{2,1} \Bmat^1_3 & \Dmat^2 &  \Qmat^2_3 \Pmat^3_{2} & \Qmat^2_3 \Wmat^3_{2,4}  \Pmat^4_3 \\
3 & \Cmat^3_1 \Bmat^1_3 &  \Cmat^3_2 \Bmat^2_3  & \Dmat^3 & \Qmat^3_4 \Pmat^4_3   \\
4 & \Cmat^4_3 \Umat^3_{4,1} \Bmat^1_3 & \Cmat^4_3 \Umat^3_{4,2} \Bmat^2_3   &  \Cmat^4_3 \Bmat^3_4 &  \Dmat^4  \\
\end{block}
\end{blockarray}.
\end{displaymath}
The spinner matrices take on the form
\begin{displaymath}
\Spinner^{1} = \begin{blockarray}{ccc}
& \mathbf{3} & \Inp  \\
\begin{block}{c[cc]} 
\mathbf{3} & 0 & \Bmat^{1}_3 \\
 \Out & \Qmat^{1}_3  & \Dmat^{1}  \\
\end{block}
\end{blockarray}, \,  \Spinner^{2} = \begin{blockarray}{ccc}
 & \mathbf{3} & \Inp  \\
\begin{block}{c[cc]} 
\mathbf{3} & 0 &  \Bmat^{2}_3 \\
\Out & \Qmat^{2}_3  & \Dmat^2  \\
\end{block}
\end{blockarray}, \,     \Spinner^{4} = \begin{blockarray}{ccc}
&  3 &   \Inp \\
\begin{block}{c[cc]} 
3 & 0  &      \Pmat^{4}_3   \\
\Out & \Cmat^{4}_3 &  \Dmat^4   \\
\end{block}
\end{blockarray},
\end{displaymath}
\begin{equation*}
    \Spinner^{3} = \begin{blockarray}{ccccc}
& 1& 2& \mathbf{4} & \Inp  \\
\begin{block}{c[cccc]} 
1 & 0 & \Vmat^{3}_{1,2}   &  \Wmat^{3}_{1,4}  & \Pmat^{3}_1  \\
2 & \Vmat^{3}_{2,1}    & 0 & \Wmat^{3}_{2,4}    &  \Pmat^{3}_2  \\
\mathbf{4} & \Umat^{3}_{4,1}    & \Umat^{3}_{4,2}    & 0 & \Bmat^3_4  \\
\Out & \Cmat^{3}_1 & \Cmat^{3}_2 & \Qmat^{3}_4  & \Dmat^4  \\
\end{block}
\end{blockarray}, 
\end{equation*}
where the parent nodes of each spinner matrix are highlighted in bold.

\begin{figure}
    \centering
   \begin{tikzpicture}
  \node[circle, draw, fill=black, inner sep=1.5pt] (myNode) at (0,0) {};
  \node[right] at (myNode.east) {$k$}; 
  \draw[-] (myNode) -- ++(90:1.2cm) node[midway, above, pos=1.0] {$(k,j)$};
  \draw[-] (myNode) -- ++(215:2.25cm) node[midway, below, pos=1.0] {$(k,i_1)$};
    \draw[-] (myNode) -- ++(240:1.5cm) node[midway, below, pos=1.0] {$(k,i_2)$};
  \draw[-] (myNode) -- ++(325:2.25cm) node[midway, below, pos=1.0] {$(k,i_p)$};
  \node at (.5,-1) {\ldots};
\end{tikzpicture}
    \caption{Node $k\in\NodeSet$ with its children $i_1,i_2,\ldots,i_p \in \Children(k)$ and parent $j = \Parent(k)$.}
    \label{fig:spinnerillustration}
\end{figure}

\subsection{HSS and SSS matrices as special subcases} \label{sec:HSSandSSSandTQS}
While SSS and HSS matrices have their origins in disparately different fields. TQS matrices bring them under one roof as it includes SSS matrices and HSS matrices as special subcases. To further elaborate on this, consider line graph
  
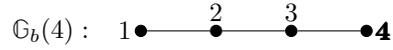
\begin{equation*}
       \begin{tikzpicture}
     \node at (-1.2,0) {$\Graph_b(4):$};
     \node[circle, draw, fill=black, inner sep=1.5pt] (myNodeA) at (0,0) {};
  \node[left] at (myNodeA.center) {$1$};
  \node[circle, draw, fill=black, inner sep=1.5pt] (myNodeB) at (1,0) {};
  \node[above] at (myNodeB.center) {$2$};
   \node[circle, draw, fill=black, inner sep=1.5pt] (myNodeC) at (2,0) {};
  \node[above] at (myNodeC.center) {$3$}; 
    \node[circle, draw, fill=black, inner sep=1.5pt] (myNodeD) at (3,0) {};
  \node[right] at (myNodeD.center) {$\pmb{\bm{4}}$}; 
    \draw (myNodeA) -- (myNodeB);
    \draw (myNodeB) -- (myNodeC);
    \draw (myNodeC) -- (myNodeD);
\end{tikzpicture}  
\end{equation*}
with node 4 chosen as the root node. The spinner matrices associated with $\mathbb{G}_b(4)$ are of the form
\begin{displaymath}
    \Spinner^{1} = \begin{blockarray}{ccc}
& \mathbf{2} & \Inp  \\
\begin{block}{c[cc]} 
\mathbf{2} & 0 & \Bmat^{1}_2 \\
 \Out & \Qmat^{1}_2  & \Dmat^{1}  \\
\end{block}
\end{blockarray}, \quad \Spinner^{2} = \begin{blockarray}{cccc}
& 1 & \mathbf{3} & \Inp  \\
\begin{block}{c[ccc]} 
1 & 0 &  \Wmat^{2}_{1,3}     & \Pmat^{2}_1 \\
\mathbf{3} &  \Umat^{2}_{3,1}        & 0 & \Bmat^{2}_3 \\
 \Out & \Cmat^{2}_1   &  \Qmat^{2}_3  & \Dmat^{2}  \\
\end{block}
\end{blockarray},
\end{displaymath}
\begin{displaymath}
  \Spinner^{3} = \begin{blockarray}{cccc}
& 2 & \mathbf{4} & \Inp  \\
\begin{block}{c[ccc]} 
2 & 0 &  \Wmat^{3}_{2,4}     & \Pmat^{3}_2 \\
\mathbf{4} &  \Umat^{3}_{4,2}        & 0 & \Bmat^{2}_4 \\
 \Out & \Cmat^{3}_2  &  \Qmat^{3}_4  & \Dmat^{3}  \\
\end{block}
\end{blockarray}, \quad  \Spinner^{4} = \begin{blockarray}{ccc}
& 3 & \Inp  \\
\begin{block}{c[cc]} 
3 & 0 & \Pmat^{4}_3 \\
 \Out & \Cmat^{4}_3  & \Dmat^{4}  \\
\end{block}
\end{blockarray},
\end{displaymath}
and generate the SSS matrix shown in \eqref{eq:SSSSn4}. In general, all linearly ordered line graphs produce a TQS matrix of the type described in \Cref{sec:SSS} if the final node is chosen as the root node.

HSS matrices, on the other hand, are TQS matrices associated with binary trees. For example, consider the post-ordered binary tree
 
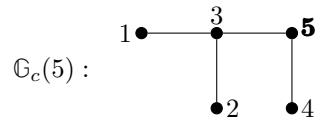
\begin{equation*}
    \begin{tikzpicture}
     \node at (-1.2,0) {$\Graph_c(5):$};
     \node[circle, draw, fill=black, inner sep=1.5pt] (myNodeA) at (0,0.5) {};
  \node[left] at (myNodeA.center) {$1$};
  \node[circle, draw, fill=black, inner sep=1.5pt] (myNodeB) at (1,0.5) {};
  \node[above] at (myNodeB.center) {$3$};
   \node[circle, draw, fill=black, inner sep=1.5pt] (myNodeC) at (2,0.5) {};
  \node[right] at (myNodeC.north) {$\pmb{\bm{5}}$}; 
    \node[circle, draw, fill=black, inner sep=1.5pt] (myNodeD) at (2,-0.5) {};
  \node[right] at (myNodeD.center) {$4$}; 
      \node[circle, draw, fill=black, inner sep=1.5pt] (myNodeE) at (1,-0.5) {};
  \node[right] at (myNodeE.center) {$2$}; 
    \draw (myNodeA) -- (myNodeB);
    \draw (myNodeB) -- (myNodeC);
    \draw (myNodeC) -- (myNodeD);
     \draw (myNodeB) -- (myNodeE);
\end{tikzpicture}  
\end{equation*}
with node $5$ as the root node. Furthermore, suppose that $n_3=m_3=n_5=m_5=0$, i.e., all the non-leaf nodes are ``empty nodes'' with \emph{zero} input and output dimensions. The spinner matrices of $\mathbb{G}_c(5)$ are of the form
\begin{displaymath}
\Spinner^{1} = \begin{blockarray}{ccc}
& \mathbf{3} & \Inp  \\
\begin{block}{c[cc]} 
\mathbf{3} & 0 & \Bmat^{1}_3 \\
 \Out & \Qmat^{1}_3  & \Dmat^{1}  \\
\end{block}
\end{blockarray},\quad  \Spinner^{2} = \begin{blockarray}{ccc}
 & \mathbf{3} & \Inp  \\
\begin{block}{c[cc]} 
\mathbf{3} & 0 &  \Bmat^{2}_3 \\
\Out & \Qmat^{2}_3  & \Dmat^2  \\
\end{block}
\end{blockarray}, \quad     \Spinner^{4} = \begin{blockarray}{ccc}
&  \mathbf{5} &   \Inp \\
\begin{block}{c[cc]} 
\mathbf{5} & 0 &      \Bmat^{4}_5   \\
\Out & \Qmat^{4}_5 &  \Dmat^4   \\
\end{block}
\end{blockarray},
\end{displaymath}
\begin{displaymath}
     \Spinner^{3} = \begin{blockarray}{ccccc}
& 1 & 2 & \mathbf{5} & \Inp  \\
\begin{block}{c[cccc]} 
1 & 0  &  \Vmat^{3}_{1,2}  & \Wmat^{3}_{1,5} & | \\
2 &  \Vmat^{3}_{2,1}  & 0 &  \Wmat^{3}_{2,5}  & | \\
\mathbf{5} &  \Umat^{3}_{5,1}  &  \Umat^{3}_{5,2}    & 0 & | \\
 \Out & - & -  &  -  & \cdot  \\
\end{block}
\end{blockarray}, \quad   \Spinner^{5} = \begin{blockarray}{ccccc}
& 3 & 4 &  \Inp  \\
\begin{block}{c[cccc]} 
3 & 0  &  \Vmat^{5}_{3,4}  &  | \\
4 &  \Vmat^{5}_{4,3}  & 0 &  | \\
 \Out & -  & -  &  \cdot  \\
\end{block}
\end{blockarray},
\end{displaymath}
and generate the HSS matrix in \eqref{eq:HSSexample}. In general, all post-ordered binary tree graphs produce a TQS matrix of the type described in \Cref{sec:HSS}, provided that all non-leaf nodes are empty nodes of zero dimensions.

In all other cases, the TQS matrices are neither SSS nor HSS. The example of the previous section and the upcoming example in \Cref{sec:constructionexample} belong to this category.

\subsection{TQS representations of sparse matrices}  \label{sec:sparsematrices}
 The construction (or realization) of a TQS representation will be treated in \Cref{sec:construction}, however, for sparse matrices whose adjacency graph is a tree, the TQS representation can be written down directly from inspection. For example, consider the family of sparse matrices whose adjacency graph corresponds to a regular $k$-level binary tree (see \Cref{fig:level3regularbinary}). Using a ``nested dissection''-styled ordering, we obtain the sequence of matrices
 $$
 \newcommand*{\tempb}{}
\TQS_1 =  [d^1], \quad \TQS_2 = \left[
\begin{array}{c|cc}
d^1 &  &  p^3_{1} \\ \hline
 & \multicolumn{1}{c|}{d^2 }  &  p^3_{2} \\ \cline{2-3}
b^1_{3} & \multicolumn{1}{c|}{ b^2_{3}}  & d^3
\end{array}\right], \quad \TQS_3 = \left[
\begin{array}{ccc|cccc}
\multicolumn{1}{c|}{d^1}  &  &  p^3_{1} \\  \cline{1-3}
\multicolumn{1}{c|}{ }  & \multicolumn{1}{c|}{d^2 } &  p^3_{2} \\  \cline{2-3}
\multicolumn{1}{c|}{b^1_{3}}  &  \multicolumn{1}{c|}{ b^2_{3}}  & d^3  &      &      &          &  p^7_{3} \\ \hline
       &         &      & \multicolumn{1}{c|}{d^4}    &   \multicolumn{1}{c|}{}    & p^6_{4}    &  \\  \cline{4-5}
       &         &      &   \multicolumn{1}{c|}{}    &  \multicolumn{1}{c|}{d^5}  & p^6_{5}   &   \\  \cline{4-7}
       &         &      &  b^4_{6} &  \multicolumn{1}{c|}{b^5_{6}}  &  \multicolumn{1}{c|}{d^6}  & p^7_{6} \\   \cline{6-7}
       &         &   b^3_{7}    &       &  \multicolumn{1}{c|}{}      &  \multicolumn{1}{c|}{b^6_{7}}   & d^7  \\ 
\end{array}\right], 
$$
$$
\TQS_4 = \scalebox{0.95}{$\left[
\begin{array}{ccccccc|cccccccc}
\multicolumn{1}{c|}{d^1}  &  &  \multicolumn{1}{c|}{p^3_{1}} &  & & & & &\\  \cline{1-3}
\multicolumn{1}{c|}{ }  & \multicolumn{1}{c|}{d^2 } &  \multicolumn{1}{c|}{p^3_{2}} &  & & & & & & \\  \cline{2-3}
\multicolumn{1}{c|}{b^1_{3}}  &  \multicolumn{1}{c|}{ b^2_{3}}  & \multicolumn{1}{c|}{d^3}  &      &      &          &  p^7_{3} & \\ \cline{1-7}
       &         &   \multicolumn{1}{c|}{}   & \multicolumn{1}{c|}{d^4}    &   \multicolumn{1}{c|}{}    & p^6_{4}    &  & \\  \cline{4-5}
       &         &    \multicolumn{1}{c|}{}  &   \multicolumn{1}{c|}{}    &  \multicolumn{1}{c|}{d^5}  & p^6_{5}   &  
 & \\  \cline{4-7}
       &         &  \multicolumn{1}{c|}{}     &  b^4_{6} &  \multicolumn{1}{c|}{b^5_{6}}  &  \multicolumn{1}{c|}{d^6}  & p^7_{6}  &\\   \cline{6-7}
       &         &  \multicolumn{1}{c|}{ b^3_{7} }   &       &  \multicolumn{1}{c|}{}      &  \multicolumn{1}{c|}{b^6_{7}}   & d^7  &  & &  & & & & &  p^{15}_{7}  \\ \hline
& & & & & & &   \multicolumn{1}{c|}{ d^8} & \multicolumn{1}{c|}{ } &  p^{10}_{8} & \multicolumn{1}{c|}{ } &    \\   \cline{8-9}
& & & & & & &  \multicolumn{1}{c|}{} & \multicolumn{1}{c|}{ d^9} &  p^{10}_{9} & \multicolumn{1}{c|}{ } \\  \cline{8-11}
& & & & & & & b^8_{10} & \multicolumn{1}{c|}{ b^{9}_{10}}  & \multicolumn{1}{c|}{ d^{10} } & \multicolumn{1}{c|}{ }      &      &        &  p^{14}_{10} \\  \cline{10-11}
& & & & & & &        &  \multicolumn{1}{c|}{ }        &    \multicolumn{1}{c|}{}   & \multicolumn{1}{c|}{ d^{11} }  &      & p^{13}_{11}    &  \\ \cline{8-15}
& & & & & & &        &         &      &  \multicolumn{1}{c|}{  }    & \multicolumn{1}{c|}{d^{12}}  & \multicolumn{1}{c|}{ p^{13}_{12} }   &   \\  \cline{12-13}
& & & & & & &        &         &      &  \multicolumn{1}{c|}{ b^{11}_{13}}  & \multicolumn{1}{c|}{b^{12}_{13}}  & \multicolumn{1}{c|}{ d^{13} }  & b^{14}_{13}  \\  \cline{12-15}
& & & & & & &        &         &    b^{10}_{14}    &  \multicolumn{1}{c|}{}    &       & \multicolumn{1}{c|}{ b^{13}_{14} }  & \multicolumn{1}{c|}{d^{14}}  & p^{15}_{14} \\   \cline{14-15}
& & & & & & b^7_{15} &        &         &        &   \multicolumn{1}{c|}{}      &       &   \multicolumn{1}{c|}{ }  &  \multicolumn{1}{c|}{b^{14}_{15}}  & d^{15}  \\ 
\end{array}\right]$} ,
$$
$\ldots$, etc. The matrices $\TQS_k$ (with $k=1,2,\ldots$) may be compressed with an HSS representation using the hierarchical partitioning delineated above. This unfortunately leads to an $O(k)$ growth in the dimensions of the HSS generators. The same growth will also occur if one opts for an SSS representation. Since $k \sim \log K$, with $K$ denoting the number of nodes, SSS and HSS will not permit a linear parameterization of $\TQS_k$ under this ordering\footnote{Note that the optimal ordering is unknown!}. However, a linear parameterization could be obtained if one chooses a TQS representation with the adjacency graph as the corresponding tree. Doing so, the \textit{root node} and \textit{leaf nodes} will have the spinner matrices\footnote{Recall that $\leftchild(l)$ (respectively, $\rightchild(l)$) denote the left (respectively, right) child of node $l$ a binary tree. }
\begin{displaymath}
\Spinner^r =   \begin{blockarray}{cccc}
& \leftchild(r) & \rightchild(r) & \Inp  \\
\begin{block}{c[ccc]} 
\leftchild(r) & 0 &  0     & p^{r}_{\leftchild(r)} \\
\rightchild(r) &  0      & 0 &  p^{r}_{\rightchild(r)} \\
 \Out & 1  &  1 & d^r  \\
\end{block}
\end{blockarray}   \quad \text{and} \quad \Spinner^{i} = \begin{blockarray}{ccc}
& \bm{\Parent(i)} & \Inp  \\
\begin{block}{c[cc]} 
 \bm{\Parent(i)} & 0 & b^{i}_{\Parent(i)} \\
 \Out & 1  & d^{i}  \\
\end{block}
\end{blockarray}
\end{displaymath}
respectively, while the \textit{interior nodes} (i.e., non-leaf and non-root nodes) take on form
\begin{displaymath}
  S_i =  \begin{blockarray}{ccccc}
& \leftchild(i) & \rightchild(i) & \bm{\Parent(i)} & \Inp  \\
\begin{block}{c[cccc]} 
\leftchild(i) & 0  & 0  & 0 & p^i_{\leftchild(i)} \\
\rightchild(i) &  0  & 0 &  0  & p^i_{\rightchild(i)} \\
\bm{\Parent(i)} &  0  &  0    & 0 &  b^i_{\Parent(i)} \\
 \Out & 1 & 1  & 1 & d^i  \\
\end{block}
\end{blockarray}.
\end{displaymath}
Notice that all the edge-to-edge operators are \emph{zero} while the edge-to-output operators are equal to \emph{one}. As we shall see in \Cref{prop:TQSinverse} in \Cref{sec:TQSalgebra}, the inverse of $\TQS_k$, which is (typically) a dense matrix, will also have a TQS representation of exactly the same dimensions. The edge-to-edge operators will no longer be zero.

It should be noted that a TQS representation with the adjacency graph as the corresponding tree will \emph{not} always yield a compact representation. A canonical example is the $K$-by-$K$ arrowhead matrix
\begin{displaymath}
    \begin{bmatrix}
    d^1 & & &    & p^{K}_1\\
     & d^2 &  &    & p^{K}_2 \\
     &     & \ddots  &   & \vdots \\
      &     &   &     d^{K-1} & p^K_{K-1}  \\
      b^{1}_K    &   b^{2}_K  & \cdots       & b^{K-1}_K & d^{K}  \\
    \end{bmatrix}.
\end{displaymath}
The TQS representation for the arrowhead matrix will involve roughly the same number of parameters as the dense matrix itself, while an SSS (or HSS) representation yields a linear parameterization. In general, one can expect a TQS representation to be efficient if
\begin{equation}
\rho_{\max} := \max_{e\in\EdgeSet} \weight{e}\ll M,N, \quad   \degnode_{\max} :=  \max_{i\in\NodeSet} \degnode(i) \ll K, \quad M,N \sim K. \label{eq:linearcomplexity}
\end{equation}

\begin{figure}
    \centering
\begin{tikzpicture}[
  level distance=1.5cm,
  level 1/.style={sibling distance=4cm},
  level 2/.style={sibling distance=2cm},
  level 3/.style={sibling distance=1cm},
  edge from parent/.style={draw}
]
\node[circle, draw, fill=black, inner sep=1.5pt] (root) {}
  child {node[circle, draw, fill=black, inner sep=1.5pt] (l1) {}
    child {node[circle, draw, fill=black, inner sep=1.5pt] (l11) {}
      child {node[circle, draw, fill=black, inner sep=1.5pt] (l111) {}}
      child {node[circle, draw, fill=black, inner sep=1.5pt] (l112) {}}
    }
    child {node[circle, draw, fill=black, inner sep=1.5pt] (l12) {}
      child {node[circle, draw, fill=black, inner sep=1.5pt] (l121) {}}
      child {node[circle, draw, fill=black, inner sep=1.5pt] (l122) {}}
    }
  }
  child {node[circle, draw, fill=black, inner sep=1.5pt] (r1) {}
    child {node[circle, draw, fill=black, inner sep=1.5pt] (r11) {}
      child {node[circle, draw, fill=black, inner sep=1.5pt] (r111) {}}
      child {node[circle, draw, fill=black, inner sep=1.5pt] (r112) {}}
    }
    child {node[circle, draw, fill=black, inner sep=1.5pt] (r12) {}
      child {node[circle, draw, fill=black, inner sep=1.5pt] (r121) {}}
      child {node[circle, draw, fill=black, inner sep=1.5pt] (r122) {}}
    }
  };

\node[above] at (root.north) {$\pmb{\bm{15}}$};
\node[above left] at (l1) {$7$};
\node[above left] at (l11) {$3$};
\node[below] at (l111) {$1$};
\node[below] at (l112) {$2$};
\node[above right] at (l12) {$6$};
\node[below] at (l121) {$4$};
\node[below] at (l122) {$5$};
\node[above right] at (r1) {$14$};
\node[above left] at (r11) {$10$};
\node[below] at (r111) {$8$};
\node[below] at (r112) {$9$};
\node[above right] at (r12) {$13$};
\node[below] at (r121) {$11$};
\node[below] at (r122) {$12$};
\end{tikzpicture}
    \caption{A level $k=3$ binary tree with a ``nested dissection''-styled ordering of the nodes.}
    \label{fig:level3regularbinary}
\end{figure}

\section{Algebraic properties of TQS matrices}  \label{sec:TQSalgebraic}
This section discusses important algebraic properties of TQS matrices. \Cref{sec:GIRSTQS} describes the graph-induced rank structure (GIRS) property of TQS matrices.  \Cref{sec:universality} introduces the theorem that characterizes the minimal TQS representation for any graph-partitioned matrix whose graph is a tree. \Cref{sec:TQSalgebra} discusses the properties of TQS matrices under addition, products, and inversion. 

\subsection{Graph-induced rank structure of TQS matrices} \label{sec:GIRSTQS}

Inverses of tridiagonal matrices are examples of dense matrices whose off-diagonal blocks possess low-rank qualities. For graph-partitioned matrices $(\TQS,\Graph)$, this property of off-diagonal low rank can be expressed in a more general framework through the notion of \emph{graph-induced rank structure} (GIRS). To describe GIRS, we first introduce the concept of a Hankel block. Let $\NodeSubset{A} \subset \NodeSet$ and let $\bar{\NodeSubset{A}} = \NodeSet \setminus \NodeSubset{A}$ denote its complement. We may (block-)permute $\TQS$ such that
\begin{displaymath}
    \mat{\Pi}_1 \TQS \mat{\Pi}_2 = \begin{bmatrix}
    \TQS\nodeid{\NodeSubset{A}, \NodeSubset{A}} &  \TQS\nodeid{\NodeSubset{A}, \bar{\NodeSubset{A}}}   \\
     \TQS\nodeid{\bar{\NodeSubset{A}}, \NodeSubset{A}}  & \TQS\nodeid{\bar{\NodeSubset{A}}, \bar{\NodeSubset{A}}}
    \end{bmatrix}.
\end{displaymath}
We call  $\TQS\nodeid{\bar{\NodeSubset{A}}, \NodeSubset{A}}$ the \emph{Hankel block induced by $\NodeSubset{A}$}. An edge $(i,j) \in \EdgeSet$ is called a \emph{border edge} with respect to $\NodeSubset{A}$  if $i\in \NodeSubset{A}$ and $j\in \bar{\NodeSubset{A}}$. The number of border edges, or the \emph{edge count}, induced by $\NodeSubset{A}$ is denoted $\EdgeCount(\NodeSubset{A})$. The GIRS property is defined as follows.
\begin{definition}[GIRS-property]
    The pair $(\TQS,\Graph)$ is said to satisfy the graph-induced (low-)rank structure for a constant $c\geq 0$ if $\forall \NodeSubset{A} \subset \NodeSet$, we have
    \begin{displaymath}
        \rank \TQS\nodeid{\bar{\NodeSubset{A}}, \NodeSubset{A}}  \leq c \EdgeCount(\NodeSubset{A}).
    \end{displaymath}
\end{definition}
TQS matrices turn out to satisfy a GIRS-property. To see this, one must exploit the direct correspondence between the path followed from node $j\in\NodeSet$ to $i\in\NodeSet$ in $\Graph$ and the associated matrix expression at block-entry $\TQS\nodeid{i,j}$. 
%
%
We may prove the following result.
\begin{proposition} \label{prop:TQSisGIRS} A TQS matrix $\TQS\in\field^{M\times N}$ with rank-profile $\lbrace \weight{e} \rbrace_{e\in \EdgeSet}$ satisfies the GIRS-property for $c=\max_{e\in \EdgeSet} \weight{e}$.
\end{proposition}
\begin{proof}
Let $\NodeSubset{A} \subset \NodeSet$ and $\bar{\NodeSubset{A}} = \NodeSet \setminus \NodeSubset{A}$. We must show that the rank of $\TQS\nodeid{\bar{\NodeSubset{A}}, \NodeSubset{A}}$ is bounded by $\EdgeCount\left( \NodeSubset{A} \right) \cdot \max_{e\in \EdgeSet} \weight{e}$. Suppose $i\in \NodeSubset{A}$ and $j\in \bar{\NodeSubset{A}}$. Since $\Graph$ is acyclic, there exists a unique path $\Path(i,j)$ connecting $i$ to $j$. By construction, $\Path(i,j)$ first starts at a node in $\NodeSubset{A}$ to eventually transition into a node in $\bar{\NodeSubset{A}}$. It may be possible that $\Path(i,j)$ leaves and enters $\NodeSubset{A}$ multiple times before it finally enters back into $\bar{\NodeSubset{A}}$ one last time to reach node $j$. Write 
$$\Path(i,j) = i-\cdots-s-t-\cdots-v-w-\cdots-j,$$
where $(s,t)\in \EdgeSet$ (with $s\in \NodeSubset{A}$, $t \in \bar{\NodeSubset{A}}$) marks the first time $\Path(i,j)$ entering $\bar{\NodeSubset{A}}$ and $(v,w)\in \EdgeSet$ (with $s\in \NodeSubset{A}$, $t \in \bar{\NodeSubset{A}}$) marks the last time  $\Path(i,j)$ leaving  $\NodeSubset{A}$. We may factor
\begin{equation}
    \TQS\nodeid{j, i} =  \mat{\Gamma}^{j}_{(v,w)} \mat{\Phi}_{(v,w),(s,t)}    \mat{\Psi}^{i}_{(s,t)}  \label{eq:keyproperty}
\end{equation}
where
\begin{eqnarray*}
   \mat{\Psi}^{i}_{(s,t)}   & := & \begin{cases}
   \Inp^s_t  & i=s \\
    \Trans^s_{t,*} \cdots  \Trans^*_{*,j} \Inp^i_*  & i\ne s
 \end{cases},  \\
\mat{\Phi}_{(v,w),(s,t)}    & := & \begin{cases}
   \Id & (s,t)=(v,w) \\
\Trans^v_{w,*} \cdots  \Trans^t_{*,s}  &  (s,t) \ne (v,w) 
 \end{cases}, \\
   \mat{\Gamma}^{j}_{(v,w)}   & := & \begin{cases}
   \Out^w_v  & j=w \\
   \Out^j_* \Trans^*_{j,*} \cdots  \Trans^v_{*,w}  & j\ne t
 \end{cases}. 
\end{eqnarray*}
 Let $\lbrace e_i \rbrace^{\EdgeCount(A)}_{i=1}\subset \EdgeSet$ denote the set of border edges, and 
\begin{eqnarray*}
  \Phi :=  \begin{bmatrix}
          \Phi_{e_1,e_1}  &  \cdots & \Phi_{e_{1},e_{\EdgeCount(A)}}  \\
        \vdots &     & \vdots \\
        \Phi_{e_{\EdgeCount(A)},e_1}   & \cdots & \Phi_{e_{\EdgeCount(A)},e_{\EdgeCount(A)}}  
    \end{bmatrix}.
\end{eqnarray*}
Given \eqref{eq:keyproperty}, it becomes evident that we may factor $\TQS\nodeid{\bar{\NodeSubset{A}}, \NodeSubset{A}} = \mat{\Gamma} \Phi \mat{\Psi}$. Thus,
\begin{displaymath}
    \rank\TQS\nodeid{\bar{\NodeSubset{A}}, \NodeSubset{A}} = \sum^{ \EdgeCount(\NodeSubset{A})}_{i=1} \weight{e_i} \leq  \EdgeCount\left( \NodeSubset{A} \right) \cdot \max_{e\in \EdgeSet} \weight{e}.  
\end{displaymath}
\end{proof}

\subsection{Minimal TQS representations} \label{sec:universality}
\Cref{prop:TQSisGIRS} shows that TQS matrices satisfy the GIRS-property for $c=\max_{e\in \EdgeSet} \weight{e}$. A question arises as to whether the converse also holds. If a tree-graph-partitioned matrix satisfies the GIRS-property for $c>0$, does this then imply the existence of a TQS representation whose rank-profile is bounded by the GIRS constant? For SSS and HSS matrices both these questions can be answered in the affirmative. Interestingly, the same holds also for the more general TQS matrices. The answer this question, one must study the problem of constructing a minimal TQS representation.
\begin{definition}[minimal TQS representation]
      Let $(\TQS,\Graph)$ be a graph-partitioned matrix with $\Graph$  acyclic and connected. A TQS representation for $\TQS$ with rank-profile $\lbrace \rho_{e} \rbrace_{e\in\EdgeSet}$ is called minimal if any other TQS representation for $\TQS$ with rank-profile $\lbrace \rho'_{e} \rbrace_{e\in\EdgeSet}$ satisfies $\rho'_{e} \geq \weight{e}$ for all $e\in\EdgeSet$. 
\end{definition}

Given a graph-partitioned matrix $(\TQS,\Graph)$ with $G$ acyclic and connected, the rank-profile of the corresponding TQS representation can be derived from the ranks of Hankel blocks whose edge count is unity. For a tree, these so-called \emph{unit Hankel blocks} are quite straightforward to enlist as every edge $(i,j)\in \EdgeSet$ can be uniquely paired with one such unit Hankel blocks. Indeed, once a root node $r\in\NodeSet$ for the graph has been picked, it must hold that either $j=\Parent(i)$ or $i=\Parent(j)$ for the corresponding tree $\mathbb{G}(r)$. Let
\begin{equation*}
    \NodeSubset{H}_{(i,j)} = \begin{cases}
      \Descendants(i) & j=\Parent(i) \\
        \NodeSet \, \setminus  \Descendants(j) & i=\Parent(j) 
    \end{cases}
\end{equation*}
and let $\UnitHankelBlock{(i,j)}$ denote the Hankel block corresponding with subset $\NodeSubset{H}_{(i,j)}$, i.e., $\UnitHankelBlock{(i,j)} = \TQS\nodeid{\NodeSubset{H}_{(i,j)} , \bar{\NodeSubset{H}}}_{(i,j)}$ with $\bar{\NodeSubset{H}}_{(i,j)} = \NodeSet \setminus \NodeSubset{H}_{(i,j)}$. We have the following result.

\begin{theorem} \label{thm:universality}
  Let $(\TQS,\Graph)$ be a graph-partitioned matrix with $\Graph$ acyclic and connected. Then $\TQS\in\field^{M\times N}$ admits an TQS representation with rank-profile 
  \begin{displaymath}
      \weight{e} = \rank \UnitHankelBlock{e}, \qquad e\in \EdgeSet.
  \end{displaymath}
  Furthermore, such a TQS representation is minimal.
\end{theorem}
The proof of \Cref{thm:universality} is postponed to \Cref{sec:proofuniversality} where we shall introduce an explicit algorithm to convert a dense matrix into a TQS representation.
\Cref{prop:TQSisGIRS} and \Cref{thm:universality} yields the following corollary.   
\begin{corollary}
      Let $(\TQS,\Graph)$ be a graph-partitioned matrix with $\Graph$  acyclic and connected. Then $\TQS\in\field^{M\times N}$ satisfies the GIRS-property for $c>0$ if, and only if, $\TQS\in\field^{M\times N}$ admits a TQS representation with a rank-profile $\lbrace \rho_{e} \rbrace_{e\in\EdgeSet}$ satisfying $\weight{e}\leq c$ for all $e\in \EdgeSet$.
\end{corollary}

\subsection{Sums, products, and inverses of TQS matrices} \label{sec:TQSalgebra}
The algebraic properties of SSS and HSS matrices under addition, multiplication, and inversion generalize to TQS matrices. The following three propositions may be established from \Cref{thm:universality}.
\begin{proposition}[TQS addition]
     Let $\TQS_1, \TQS_2 \in \field^{M\times N}$, with $M=\sum_{i\in \NodeSet} m_i$ and $N=\sum_{i\in \NodeSet} n_i$, be TQS matrices of rank-profiles  $\lbrace \weight{1,e} \rbrace_{e\in \EdgeSet}$  and 
     $\lbrace \weight{2,e} \rbrace_{e\in \EdgeSet}$ associated with the acyclic connected graph $\Graph =(\NodeSet,\EdgeSet)$. Then, $\TQS_3 = \TQS_1 + \TQS_2$ is a TQS matrix of rank-profile $\lbrace \weight{1,e} + \weight{2,e} \rbrace_{e\in \EdgeSet}$.
\end{proposition}
\begin{proof}
    Since $\rank \TQS_3\nodeid{\bar{\NodeSubset{A}}, \NodeSubset{A}} \leq \rank \TQS_1\nodeid{\bar{\NodeSubset{A}}, \NodeSubset{A}} +  \rank \TQS_2\nodeid{\bar{\NodeSubset{A}}, \NodeSubset{A}}$  for any Hankel block induced by $\NodeSubset{A}$, the result directly follows from \Cref{thm:universality}. 
\end{proof}

\begin{proposition}[TQS product]
      Let $\TQS_1\in \field^{M\times N}$  and $\TQS_2 \in \field^{N\times P}$, with $M=\sum_{i\in \NodeSet} m_i$,  $N=\sum_{i\in \NodeSet} n_i$, and $P=\sum_{i\in \NodeSet} p_i$, be TQS matrices of rank-profiles  $\lbrace \weight{1,e} \rbrace_{e\in \EdgeSet}$  and 
     $\lbrace \weight{2,e} \rbrace_{e\in \EdgeSet}$ associated with the acyclic connected graph $\Graph =(\NodeSet,\EdgeSet)$. Then, $\TQS_3 = \TQS_1 \TQS_2$ is a TQS matrix of rank-profile $\lbrace \weight{1,e} + \weight{2,e} \rbrace_{e\in \EdgeSet}$.
\end{proposition}
\begin{proof}
     Since 
     \begin{eqnarray*}
         \rank \TQS_3\nodeid{\bar{\NodeSubset{A}}, \NodeSubset{A}} & = & \rank \left( \TQS_1\nodeid{\bar{\NodeSubset{A}}, \NodeSubset{A}} \TQS_2\nodeid{\NodeSubset{A}, \NodeSubset{A}}  + \TQS_1\nodeid{\bar{\NodeSubset{A}}, \bar{\NodeSubset{A}}} \TQS_2\nodeid{\bar{\NodeSubset{A}}, \NodeSubset{A}} \right) \\
         & \leq & \rank \TQS_1\nodeid{\bar{\NodeSubset{A}}, \NodeSubset{A}} +  \rank \TQS_2\nodeid{\bar{\NodeSubset{A}}, \NodeSubset{A}} 
     \end{eqnarray*}
      for any Hankel block induced by $\NodeSubset{A}$, the result directly follows from \Cref{thm:universality}. 
\end{proof}

\begin{proposition}[TQS inverse]\label{prop:TQSinverse}
       Let $\TQS \in \field^{N\times N}$, with $N=\sum_{i\in \NodeSet} n_i$, be a non-singular TQS matrix of rank-profile  $\lbrace \weight{e} \rbrace_{e\in \EdgeSet}$ associated with the acyclic connected graph $\Graph =(\NodeSet,\EdgeSet)$. Then, $\TQS^{-1}$ is also a TQS matrix of rank-profile $\lbrace \weight{e} \rbrace_{e\in \EdgeSet}$.
\end{proposition}
\begin{proof}
 Consider the Hankel block $\TQS \nodeid{\NodeSubset{A}, \NodeSubset{A}}$ and note that, under the hypothesis of \cref{prop:TQSinverse},  $\TQS \nodeid{\NodeSubset{A}, \NodeSubset{A}}$ is a square matrix.  Observe that under the hypothesis,  let $\TQS \nodeid{\NodeSubset{A}, \NodeSubset{A}} = \mat{U} \Sigma \mat{V}$ denote singular value decomposition and $\mat{B}(\epsilon) = \mat{U} (\Sigma + \epsilon \Id) \mat{V}$ for $\epsilon>0$. By construction the inverse of $\mat{B}(\epsilon)$ exists, and 
$$
\begin{bmatrix}
    \mat{B}(\epsilon) &  \TQS\nodeid{\NodeSubset{A}, \bar{\NodeSubset{A}}}   \\
     \TQS\nodeid{\bar{\NodeSubset{A}}, \NodeSubset{A}}  & \TQS\nodeid{\bar{\NodeSubset{A}}, \bar{\NodeSubset{A}}}
\end{bmatrix}^{-1} \!=\! \begin{bmatrix}
   * &  *  \\
     -\!\left( \TQS\nodeid{\bar{\NodeSubset{A}}, \bar{\NodeSubset{A}}}\!-\!    \TQS\nodeid{\bar{\NodeSubset{A}}, \NodeSubset{A}}  \mat{B}^{-1}\!(\epsilon)   \TQS\nodeid{\NodeSubset{A}, \bar{\NodeSubset{A}}}
 \right)^{-1}\!\TQS\nodeid{\bar{\NodeSubset{A}}, \NodeSubset{A}}  \mat{B}^{-1}\!(\epsilon)  & * 
\end{bmatrix}^{-1}\!.
$$
By taking limits for $\epsilon\rightarrow 0$, it becomes straightforward to show that $\rank \TQS^{-1}\nodeid{\bar{\NodeSubset{A}}, \NodeSubset{A}} = \rank \TQS \nodeid{\bar{\NodeSubset{A}}, \NodeSubset{A}}$ for any Hankel block induced by $\NodeSubset{A}$. The proposition then follows from \Cref{thm:universality} and a limiting argument on the rank of determinants and minors. 
\end{proof}
\section{TQS matrix construction} \label{sec:construction}
This section discusses the construction (or realization in the language of systems theory) of a TQS representation from a dense matrix provided a tree $\Graph(r)$ and an accompanying partitioning of the matrix. In \Cref{sec:constructionexample} we describe the construction on an illustrative example. The general algorithm is described in \Cref{sec:constructionalg}. The construction or realization algorithm is naturally a generalization of the SSS and HSS realization algorithms. The presented algorithm will allow us to prove \cref{thm:universality} in \cref{sec:universality}. This is done in \Cref{sec:proofuniversality}.

\subsection{An illustrative example} \label{sec:constructionexample}
Before describing the general algorithm, we first illustrate the construction of a TQS representation on an illustrative example. Consider the tree
 
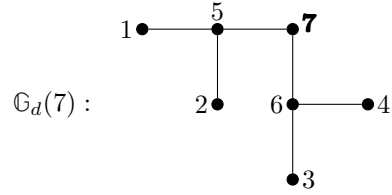
\begin{equation*}
\begin{tikzpicture}
     \node at (-1.2,0) {$\Graph_d(7):$};
     \node[circle, draw, fill=black, inner sep=1.5pt] (myNodeA) at (0,1.0) {};
  \node[left] at (myNodeA.center) {$1$};
  \node[circle, draw, fill=black, inner sep=1.5pt] (myNodeB) at (1,1) {};
  \node[above] at (myNodeB.center) {$5$};
   \node[circle, draw, fill=black, inner sep=1.5pt] (myNodeC) at (2,1) {};
  \node[right] at (myNodeC.north) {$\pmb{\bm{7}}$}; 
\node[circle, draw, fill=black, inner sep=1.5pt] (myNodeE) at (1,0) {};
  \node[left] at (myNodeE.center) {$2$}; 
    \node[circle, draw, fill=black, inner sep=1.5pt] (myNodeD) at (2,0) {};
  \node[left] at (myNodeD.center) {$6$}; 
   \node[circle, draw, fill=black, inner sep=1.5pt] (myNodeF) at (3,0) {};
  \node[right] at (myNodeF.center) {$4$}; 
     \node[circle, draw, fill=black, inner sep=1.5pt] (myNodeG) at (2,-1) {};
  \node[right] at (myNodeG.center) {$3$}; 
    \draw (myNodeA) -- (myNodeB);
    \draw (myNodeB) -- (myNodeC);
    \draw (myNodeC) -- (myNodeD);
     \draw (myNodeB) -- (myNodeE);
     \draw (myNodeD) -- (myNodeG);
     \draw (myNodeD) -- (myNodeF);
\end{tikzpicture}  
\end{equation*}
with node $7$ picked as the root node. $\mathbb{G}_b(7)$ is a tree of depth 2 and comprises of the levels: $\NodeSet_0 = \lbrace 7 \rbrace$,  $\NodeSet_1 = \lbrace 5,6 \rbrace$, and $\NodeSet_2 = \lbrace 1,2,3,4 \rbrace$.   The corresponding spinner matrices and TQS form are shown in \Cref{fig:TQSexample}. It turns out that the generating matrices of the TQS representation can be retrieved from computing low-rank factorizations of the unit Hankel blocks $\UnitHankelBlock{(i,j)}$ in a particular sequence. 

\begin{sidewaysfigure}
    \centering
\begin{equation*}
\Spinner^{1} = \begin{blockarray}{ccc}
& \mathbf{5} & i  \\
\begin{block}{c[cc]} 
 \mathbf{5} & 0 &  \Bmat^{1}_5 \\
 o & \Qmat^{1}_5  & \Dmat^1   \\
\end{block}
\end{blockarray}\quad \Spinner^{2} = \begin{blockarray}{ccc}
& \mathbf{5} & i  \\
\begin{block}{c[cc]} 
\mathbf{5} & 0 &  \Bmat^{2}_5 \\
o & \Qmat^{2}_5  & \Dmat^2   \\
\end{block}
\end{blockarray}\quad \Spinner^{3} = \begin{blockarray}{ccc}
& \mathbf{6} & i  \\
\begin{block}{c[cc]} 
0 &  \Bmat^{3}_6 & \mathbf{6}\\
o & \Qmat^{3}_6  & \Dmat^3   \\
\end{block}
\end{blockarray}, \quad \Spinner^{4} = \begin{blockarray}{ccc}
& \mathbf{6} & i  \\
\begin{block}{c[cc]} 
\mathbf{6} & 0 &  \Bmat^{4}_6 \\
 o & \Qmat^{4}_6  & \Dmat^4  \\
\end{block}
\end{blockarray}
\end{equation*}
\begin{equation*}
    \Spinner^{5} = \begin{blockarray}{ccccc}
& 1& 2& \mathbf{7} & i  \\
\begin{block}{c[cccc]} 
1 & 0 &    \Vmat^{5}_{1,2}   &   \Wmat^{5}_{1,7}  &  \Pmat^{5}_1   \\
2 & \Vmat^{5}_{2,1}    &   0  &   \Wmat^{5}_{2,7}    &   \Pmat^{5}_2   \\
 \mathbf{7} & \Umat^{5}_{7,1}    &   \Umat^{5}_{7,2}    &    0 &   \Bmat^5_7    \\
 o & \Cmat^{5}_1  &  \Cmat^{5}_2 &   \Qmat^{5}_7  & \Dmat^5  \\
\end{block}
\end{blockarray}, \quad \Spinner^{6} = \begin{blockarray}{ccccc}
& 3& 4& \mathbf{7} & i  \\
\begin{block}{c[cccc]} 
3 & 0 &  \Vmat^{6}_{3,4}   &    \Wmat^{6}_{3,7}  &   \Pmat^{6}_3   \\
4 & \Vmat^{6}_{4,3}    &    0&   \Wmat^{6}_{4,7}    &     \Pmat^{6}_{4}  \\
\mathbf{7} & \Umat^{6}_{7,3}    & \Umat^{6}_{7,4}    & 0 &  \Bmat^6_7   \\
o & \Cmat^{6}_3 &  \Cmat^{6}_4 &   \Qmat^{6}_7  & \Dmat^6  \\
\end{block}
\end{blockarray}
\end{equation*}
\begin{equation*}
    \Spinner^{7} = \begin{blockarray}{cccc}
& 5& 6&  i \\
\begin{block}{c[ccc]} 
5 & 0 &  \Wmat^{7}_{5,6}    &  \Pmat^{7}_5  \\
6 & \Wmat^{7}_{6,5}    &    0    &   \Pmat^{7}_6  \\
 o & \Cmat^{7}_5 &  \Cmat^{7}_6 & \Dmat^7 \\
\end{block}
\end{blockarray}
\end{equation*}
 \resizebox{\textwidth}{!}{
$\TQS = \begin{blockarray}{cccccccc}
1 & 2 & 3 & 4 & 5 & 6 & 7 \\
\begin{block}{[ccccccc]c} 
\Dmat^1 &   \Qmat^1_5 \Vmat^5_{1,2} \Bmat^2_5 & \Qmat^1_5 \Wmat^5_{1,7} \Vmat^7_{5,6} \Umat^6_{7,3} \Bmat^3_6 & \Qmat^1_5 \Wmat^5_{1,7} \Vmat^7_{5,6} \Umat^6_{7,4} \Bmat^4_6 &  \Qmat^1_5 P^5_1  & \Qmat^1_{5} \Umat^5_{1,7} \Vmat^7_{5,6} \Bmat^6_7  & \Qmat^1_5 \Wmat^5_{1,7} \Pmat^7_5 & 1 \\
\Qmat^2_5 \Vmat^5_{2,1} \Bmat^1_5 & \Dmat^2 &  \Cmat^2_5 \Wmat^5_{2,7} \Vmat^7_{6,5} \Umat^6_{7,3} \Bmat^3_6 &  \Qmat^2_5 \Wmat^5_{1,7} \Vmat^7_{5,6} \Umat^6_{7,4} \Bmat^4_6 &  \Qmat^2_5 \Pmat^5_{2} & \Qmat^2_5 \Wmat^5_{2,7} \Vmat^7_{5,6} \Bmat^6_7 & \Qmat^2_5 \Wmat^5_{2,7} \Pmat^7_5 & 2  \\
\Qmat^3_6 \Wmat^6_{3,7} \Vmat^7_{6,5} \Umat^5_{7,1} \Bmat^1_5 &  \Qmat^3_6 \Wmat^6_{3,7} \Vmat^7_{6,5} \Umat^5_{7,2} \Bmat^2_5 & \Dmat^3 & \Qmat^3_6 \Vmat^6_{3,4} \Bmat^4_6 & \Qmat^3_6 \Wmat^6_{3,7} \Vmat^7_{6,5} \Bmat^5_7 & \Qmat^3_6 \Pmat^6_3 & \Qmat^3_6 \Wmat^6_{3,7} \Pmat^7_{6} & 3 \\
\Qmat^4_6 \Wmat^6_{4,7} \Vmat^7_{6,5} \Umat^5_{7,1} \Bmat^1_5 &    \Qmat^4_6 \Wmat^6_{4,7} \Vmat^7_{6,5} \Umat^5_{7,2} \Bmat^2_5 & \Qmat^4_6 \Vmat^6_{4,3} \Bmat^3_6 & \Dmat^4 & \Qmat^4_6 \Wmat^6_{4,7} \Vmat^7_{6,5} \Bmat^5_{7} &  \Qmat^4_6 \Pmat^6_4 & \Qmat^4_6 \Wmat^6_{1,7} \Pmat^7_6 & 4 \\
\Cmat^5_1 \Bmat^1_5 &  \Cmat^5_2 \Bmat^2_5 &   \Qmat^5_{7} \Vmat^7_{5,6} \Umat^6_{7,3} \Bmat^3_6 &  \Qmat^5_7 \Vmat^7_{5,6} \Umat^6_{7,4} \Bmat^4_6 & \Dmat^5 & \Qmat^5_7 \Vmat^7_{5,6} \Bmat^6_7 &  \Qmat^5_7 \Pmat^7_5 & 5 \\
\Qmat^6_7 \Vmat^7_{6,5} \Umat^5_{7,1} \Bmat^1_5 &    \Qmat^6_7 \Umat^7_{6,5} \Vmat^5_{7,2} \Bmat^2_5 &  \Cmat^6_{3} \Bmat^3_6 & \Cmat^6_4 \Bmat^4_{6} &  \Qmat^6_7 \Vmat^7_{7,5} \Bmat^5_7 & \Dmat^6 & \Qmat^6_7 \Pmat^7_6 & 6 \\
\Cmat^7_5 \Umat^5_{7,1} \Bmat^1_5 & \Cmat^7_5 \Umat^5_{7,2} \Bmat^2_5 & \Cmat^7_{6} \Umat^6_{7,3} \Bmat^3_6 & \Cmat^7_6 \Umat^6_{7,4} \Bmat^4_6 &  \Cmat^7_5 \Bmat^5_7 &  \Cmat^7_6 \Bmat^6_{7} & \Dmat^7 & 7 \\
\end{block}
\end{blockarray}
$
}
    \caption{The spinner matrices and TQS form associated with the tree $\Graph_b(7)$}
    \label{fig:TQSexample}
\end{sidewaysfigure}

To start, we begin with the unit Hankel blocks associated with the edges at the deepest level of the tree. Specifically, from the low-rank factorizations of the unit Hankel blocks $\UnitHankelBlock{(i,j)}=\Xmat{(i,j)} \Ymat{(i,j)}$ with $i\in \NodeSet_2$ and $j=\Parent(i)\in \NodeSet_1$, we shall be able to obtain the $\Bmat$'s of the spinner matrices corresponding to the nodes in $\NodeSet_2$ and the $\Cmat$'s of the spinner matrices corresponding to the nodes in $\NodeSet_1$. For example, for $i=1 \in \NodeSet_2$ and $j= \Parent(1)=5 \in \NodeSet_1$, we may write
\begin{displaymath}
 \UnitHankelBlock{(1,5)} = \begin{blockarray}{cc}\begin{block}{c[c]}
    2 &  \Xmat{(1,5)}\nodeid{2} \\
    3 & \Xmat{(1,5)}\nodeid{3}  \\
    4 & \Xmat{(1,5)} \nodeid{4} \\
    5 & \Xmat{(1,5)} \nodeid{5}  \\
    6 & \Xmat{(1,5)} \nodeid{6}  \\
    7 & \Xmat{(1,5)} \nodeid{7}  \\
 \end{block} 
 \end{blockarray} \begin{blockarray}{c}
 1 \\
     \begin{block}{[c]}
         \Ymat{(1,5)}\nodeid{1} \\
     \end{block}
 \end{blockarray}
\end{displaymath}
since we know that $\UnitHankelBlock{(1,5)}$ should factor into  
\begin{displaymath}
\begin{blockarray}{cc}
\begin{block}{c[c]} 
2 & \Qmat^2_5 \Vmat^5_{2,1}  \\ 
3 & \Qmat^3_6 \Wmat^6_{3,7} \Vmat^7_{6,5} \Umat^5_{7,1}   \\
4 & \Qmat^4_6 \Wmat^6_{4,7} \Vmat^7_{6,5} \Umat^5_{7,1}  \\
5 & \Cmat^5_1  \\
6 & \Qmat^6_7 \Vmat^7_{6,5} \Umat^5_{7,1}  \\
7 & \Cmat^7_5 \Umat^5_{7,1} \\
\end{block}
\end{blockarray}
\begin{blockarray}{c}
1 \\
\begin{block}{[c]} 
B^1_5 \\
\end{block}
\end{blockarray},
\end{displaymath}
we may set $\Cmat^5_1 =  \Xmat{(1,5)}\nodeid{5}\in\field^{m_5 
\times \weight{(1,5)}}$ and $B^5_1 = \Ymat{(1,5)}\nodeid{1}\in\field^{\weight{(1,5)} \times n_1}$ with $\weight{(1,5)} = \rank\UnitHankelBlock{(1,5)}$. With similar reasoning, we may set $\Cmat^5_2 = \Xmat{(2,5)} \nodeid{5} \in \field^{m_5 \times \weight{(2,5)}}$, $\Bmat^2_5 = \Ymat{(2,5)}\nodeid{2} \in \field^{\weight{(2,5)} \times n_2}$, $\Cmat^6_{3}=\Xmat{(3,6)}\nodeid{6} \in \field^{m_6 \times \weight{(3,6)}}$, $\Bmat^3_6 = \Ymat{(3,6)}\nodeid{3} \in  \field^{\weight{(3,6)} \times n_3}$, $\Cmat^6_4 = \Xmat{(4,6)} \nodeid{6} \in \field^{ m_6 \times \weight{(4,6)} }$, $\Bmat^4_6=\Ymat{(4,6)}\nodeid{4} \in \field^{ \weight{(4,6)} \times n_4}$ with $\rho_{(2,5)} = \rank \UnitHankelBlock{(2,5)}$, $\rho_{(3,6)} = \rank \UnitHankelBlock{(3,6)}$, and $\rho_{(4,6)} = \rank  \UnitHankelBlock{(4,6)}$.

Next, moving one level up the tree by computing low-rank factorizations $\UnitHankelBlock{(i,j)}=\Xmat{(i,j)} \Ymat{(i,j)}$ with $i\in \NodeSet_1$ and $j=\Parent(i)\in \NodeSet_0$, we can obtain the $\Umat$'s and $\Bmat$'s of the spinner matrices corresponding to the nodes in $\NodeSet_1$ and the $\Cmat$'s of the spinner matrices corresponding to the nodes in $\NodeSet_0$. For example, for $i=5 \in \NodeSet_1$ and $j= \Parent(5)=7 \in \NodeSet_0$, we may write 
\begin{displaymath}
       \UnitHankelBlock{(5,7)}  =  \begin{blockarray}{cccc}
       \begin{block}{c[ccc]} 
      3 &     \Xmat{(1,5)} \nodeid{3}  & \Xmat{(2,5)} \nodeid{3} & \TQS \nodeid{3,5} \\
4 & \Xmat{(1,5)}\nodeid{4}  & \Xmat{(2,5)}  \nodeid{4} & \TQS\nodeid{4,5}   \\
6 & \Xmat{(1,5)} \nodeid{6} &  \Xmat{(2,5)}  \nodeid{6} & \TQS\nodeid{6,5}  \\
7 & \Xmat{(1,5)} \nodeid{7} & \Xmat{(2,5)}  \nodeid{7} & \TQS\nodeid{7,5} \\
       \end{block}
       \end{blockarray} 
       \begin{blockarray}{ccc}
         \Descendants(1)   &  \Descendants(2)  & 5  \\
       \begin{block}{[ccc]}
       \Ymat{(1,5)} &  &  \\
                                & \Ymat{(2,5)} &  \\
                             &   & \Id \\
       \end{block}
       \end{blockarray}
\end{displaymath}
with the help of previously computed factorizations. The low-rank factorization $\UnitHankelBlock{(5,7)} = \Xmat{(5,7)} \Ymat{(5,7)}$ can be obtained by compressing the matrix 
\begin{multline*}
    \begin{blockarray}{cccc}
       \begin{block}{c[ccc]} 
      3 &     \Xmat{(1,5)} \nodeid{3}  & \Xmat{(2,5)} \nodeid{3} & \TQS \nodeid{3,5} \\
4 & \Xmat{(1,5)}\nodeid{4}  & \Xmat{(2,5)}  \nodeid{4} & \TQS\nodeid{4,5}   \\
6 & \Xmat{(1,5)} \nodeid{6} &  \Xmat{(2,5)}  \nodeid{6} & \TQS\nodeid{6,5}  \\
7 & \Xmat{(1,5)} \nodeid{7} & \Xmat{(2,5)}  \nodeid{7} & \TQS\nodeid{7,5} \\
       \end{block}
       \end{blockarray} = 
       \begin{blockarray}{cc} \begin{block}{c[c]} 
     3 &   \Xmat{(5,7)} \nodeid{3}   \\
4 & \Xmat{(5,7)} \nodeid{4}   \\
6 & \Xmat{(5,7)} \nodeid{6}   \\
7 & \Xmat{(5,7)}\nodeid{7} \\
       \end{block}
       \end{blockarray} \begin{bmatrix}
           \Zmat{(5,7)}^{\Descendants(1)}   & \Zmat{(5,7)}^{\Descendants(2)} & \Zmat{(5,7)}^5
       \end{bmatrix},
\end{multline*}
which sets
\begin{multline*}
\begin{blockarray}{ccc} 
       \Descendants(1) & \Descendants(2) & 5 \\
       \begin{block}{[ccc]}
\Ymat{(5,7)} \nodeid{\Descendants(1)}  & \Ymat{(5,7)} \nodeid{\Descendants(2)} & \Ymat{(5,7)} \nodeid{5}   \\
       \end{block}
       \end{blockarray}
   \! = \!  \begin{blockarray}{ccc} 
       \Descendants(1) & \Descendants(2) & 5 \\
       \begin{block}{[ccc]}
\Zmat{(5,7)}^{\Descendants(1)} \Ymat{(1,5)}   & \Zmat{(5,7)}^{\Descendants(2)} \Ymat{2,5)} & \Zmat{(5,7)}^5  \\
       \end{block}
       \end{blockarray}.
\end{multline*}
Since $\Ymat{(1,5)} = \Bmat^1_5$, $\Ymat{(2,5)} = \Bmat^2_5$ and $\UnitHankelBlock{(5,7)}$ should factor into
\begin{eqnarray*}
\begin{blockarray}{cc}
\begin{block}{c[c]}
3 & \Qmat^3_6 \Wmat^6_{3,7} \Vmat^7_{6,5}  \\
4 & \Qmat^4_6 \Wmat^6_{4,7} \Vmat^7_{6,5}  \\
6 & \Qmat^6_7 \Vmat^7_{6,5} \\ 
7 & \Cmat^7_5 \\
\end{block}
\end{blockarray}   \begin{blockarray}{ccc}
 \Descendants(1) & \Descendants(2) & 5 \\
\begin{block}{[ccc]}
    \Umat^5_{7,1} \Bmat^1_5 & \Umat^5_{7,2} \Bmat^2_5 & \Bmat^5_7 \\
\end{block}
\end{blockarray},
\end{eqnarray*}
we see that $\Cmat^7_5 = \Xmat{(5,7)} \nodeid{7} \in \field^{m_7 \times \weight{(5,7)}}$, $\Umat^5_{7,1} = \Zmat{(5,7)}^{\Descendants(1)} \in \field^{\weight{(5,7)} \times \weight{(1,5)}}$, $\Umat^5_{7,2} = \Zmat{(5,7)}^{\Descendants(2)} \in \field^{\weight{(5,7)} \times \weight{(2,5)}}$, $\Bmat^5_7 = \Zmat{(5,7)}^5 \in \field^{ \weight{(5,7)} \times n_5}$ with $\weight{(5,7)} = \rank \UnitHankelBlock{(5,7)}$. With similar reasoning, we may set $\Cmat^7_6 = \Xmat{(6,7)}\nodeid{7} \in \field^{  m_7 \times \weight{(6,7)} }$, $\Umat^6_{7,3} = \Zmat{(6,7)}^{\Descendants(3)} \in \field^{ \weight{(6,7)} \times \weight{(3,6)} }$, $\Umat^6_{7,4} = \Zmat{(6,7)}^{\Descendants(4)} \in \field^{ \weight{(6,7)} \times \weight{(4,6)} }$, $\Bmat^6_7 = \Zmat{(6,7)}^6 \in \field^{ \weight{(6,7)} \times n_6}$ with $\weight{(6,7)} = \rank \UnitHankelBlock{(6,7)}$.

By now, we have fully climbed up the tree and arrived at the root node. In this process, we have computed all the $\Bmat$'s, $\Umat$'s, and $\Cmat$'s of the TQS representation. This was done by peeling off the terms from the low-rank factorizations of the unit Hankel blocks. To compute the remaining  $\Pmat$'s, $\Vmat$'s, $\Wmat$'s, and $\Qmat$'s, we proceed in the same way. However, the main difference is that we start at the root and work ourselves down the tree towards the leaves.  To start, through computing the low-rank factorizations $\UnitHankelBlock{(j,i)}=\Xmat{(j,i)} \Ymat{(j,i)}$ with $i\in \NodeSet_1$ and $j=\Parent(i)\in \NodeSet_0$, we will able to retrieve the $\Pmat$'s and $\Vmat$'s of the spinner matrices corresponding to the nodes in $\NodeSet_0$ and the $\Qmat$'s of the spinner matrices corresponding to the nodes in $\NodeSet_1$. For example, for $i=5 \in \NodeSet_1$ and $j= \Parent(5)=7 \in \NodeSet_0$, we may write 
\begin{displaymath}
          \UnitHankelBlock{(7,5)}  =  \begin{blockarray}{ccc}
       \begin{block}{c[cc]} 
      1 &     \Xmat{(6,7)} \nodeid{1}  & \TQS \nodeid{1,7} \\
2 & \Xmat{(6,7)}\nodeid{2}   & \TQS\nodeid{2,7}   \\
5 & \Xmat{(6,7)} \nodeid{5}  & \TQS\nodeid{5,7} \\
       \end{block}
       \end{blockarray} 
       \begin{blockarray}{cc}
         \Descendants(6)    & 7  \\
       \begin{block}{[cc]}
        \Ymat{(6,7)} &  \\
              & \Id \\
       \end{block}
       \end{blockarray}.
\end{displaymath}
Compressing the matrix 
\begin{displaymath}
         \begin{blockarray}{ccc}
       \begin{block}{c[cc]} 
      1 &     \Xmat{(6,7)} \nodeid{1}  & \TQS \nodeid{1,7} \\
2 & \Xmat{(6,7)}\nodeid{2}   & \TQS\nodeid{2,7}   \\
5 & \Xmat{(6,7)} \nodeid{5}  & \TQS\nodeid{5,7} \\
       \end{block}
       \end{blockarray}  = \begin{blockarray}{cc} \begin{block}{c[c]} 
     1 &   \Xmat{(7,5)} \nodeid{1}   \\
2 & \Xmat{(7,5)} \nodeid{2}   \\
5 & \Xmat{(7,5)} \nodeid{5}   \\
       \end{block}
       \end{blockarray} \begin{bmatrix}
           \Zmat{(7,5)}^{\Descendants(6)}  & \Zmat{(7,5)}^7
       \end{bmatrix}
\end{displaymath}
allows us to produce the  low-rank factorization $\UnitHankelBlock{(7,5)} = \Xmat{(7,5)} \Ymat{(7,5)}$ with
\begin{displaymath}
     \begin{blockarray}{cc}
         \Descendants(6)    & 7  \\
       \begin{block}{[cc]}
        \Ymat{(7,5)}\nodeid{\Descendants(6)} & \Ymat{(7,5)}\nodeid{7} \\
       \end{block}
       \end{blockarray}  =   \begin{blockarray}{cc}
         \Descendants(6)    & 7  \\
       \begin{block}{[cc]}
        \Zmat{(7,5)}^{\Descendants(6)} \Ymat{(6,7)} & \Zmat{(7,5)}^7 \\
       \end{block}
       \end{blockarray}
\end{displaymath}
Since $\Ymat{(6,7)} = \begin{blockarray}{ccc}
3 & 4 & 6 \\
\begin{block}{[ccc]}
      \Umat^6_{7,3} \Bmat^3_6 & \Umat^6_{7,4} \Bmat^4_6 & \Bmat^6_7 \\
\end{block}
\end{blockarray}$ and $\UnitHankelBlock{(7,5)}$ should factor into 
\begin{displaymath}
        \begin{blockarray}{cc}
        \begin{block}{c[c]}
1 & \Qmat^1_5 \Wmat^5_{1,7}  \\
2 & \Qmat^2_5 \Wmat^5_{2,7}  \\
5 &  \Qmat^5_{7} \\
   \end{block}
\end{blockarray} \begin{blockarray}{cc} \Descendants(6) & 7 \\ \begin{block}{[cc]}  \Vmat^7_{5,6} \begin{bmatrix}
    \Umat^6_{7,3} \Bmat^3_6 & \Umat^6_{7,4} \Bmat^4_6 & \Bmat^6_7 
\end{bmatrix}  &  \Pmat^7_5   \\ \end{block} \end{blockarray},
\end{displaymath}
we may set $\Qmat^5_{7} = \Xmat{(7,5)} \nodeid{5} \in \field^{m_5 \times \weight{(7,5)}}$, $\Vmat^7_{5,6}= \Zmat{(7,5)}^{\Descendants(6)}  \in\field^{\weight{(7,5)} \times \weight{(6,7)}}$, $\Pmat^7_5  = \Zmat{(7,5)}^7 \in\field^{\weight{(7,5)} \times n_7}$ with $\weight{(7,5)} = \rank \UnitHankelBlock{(7,5)}$. With similar reasoning, we may set $\Qmat^6_{7} = \Xmat{(7,6)} \nodeid{6} \in \field^{m_6 \times \weight{(7,5)}}$, $\Vmat^7_{5,6}= \Zmat{(7,6)}^{\Descendants(5)} \in\field^{\weight{(7,6)} \times \weight{(5,7)}}$, $\Pmat^7_6  = \Zmat{(7,6)}^7 \in\field^{\weight{(7,6)} \times n_7}$ with $\weight{(7,6)} = \rank \UnitHankelBlock{(7,6)}$.

At last, moving one level down, we reach the bottom of the tree. By computing the low-rank factorizations $\UnitHankelBlock{(j,i)}=\Xmat{(j,i)} \Ymat{(j,i)}$ with $i\in \NodeSet_2$ and $j=\Parent(i)\in \NodeSet_1$, we will be able to compute all the remaining terms of TQS representation. Specifically,  will be able to retrieve all the $\Pmat$'s, $\Wmat$'s, and $\Vmat$'s of the spinner matrices corresponding to the nodes in $\NodeSet_1$ and the $\Qmat$'s of the spinner matrices corresponding to the nodes in $\NodeSet_2$. For example, for $i=1 \in \NodeSet_2$ and $j= \Parent(1)=5 \in \NodeSet_1$, we may write
\begin{displaymath}
    \UnitHankelBlock{(5,1)} = \begin{blockarray}{cccc}
        \begin{block}{c[ccc]}
        1 & \Xmat{(7,5)}\nodeid{1} &  \Xmat{(2,5)}\nodeid{1} & \TQS\nodeid{1,5} \\
    \end{block}
    \end{blockarray} \begin{blockarray}{ccc}
        \Descendantscompl(5) &  \Descendants(2) & 5 \\
        \begin{block}{[ccc]}
            \Ymat{(7,5)} &   &  \\
               & \Ymat{(2,5)} &  \\
               &         & \Id \\
        \end{block}
    \end{blockarray}
\end{displaymath}
Compressing the matrix 
\begin{multline*}
            \begin{blockarray}{cccc}
        \begin{block}{c[ccc]}
        1 & \Xmat{(7,5)}\nodeid{1} &  \Xmat{(2,5)}\nodeid{1} & \TQS\nodeid{1,5} \\
    \end{block}
    \end{blockarray}  =  \begin{blockarray}{cc} \begin{block}{c[c]} 
     1 &   \Xmat{(5,1)} \nodeid{1}   \\
       \end{block}
       \end{blockarray} \begin{bmatrix}
          \Zmat{(5,1)}^{\Descendantscompl(5)}  &  \Zmat{(5,1)}^{\Descendants(2)} & \Zmat{(5,1)}^5
       \end{bmatrix}
\end{multline*}
allows us to produce the  low-rank factorization $\UnitHankelBlock{(7,5)} = \Xmat{(7,5)} \Ymat{(7,5)}$ with
\begin{multline*}
\begin{blockarray}{ccc} 
       \Descendantscompl(5) & \Descendants(2) & 5 \\
       \begin{block}{[ccc]}
\Ymat{(5,2)} \! \nodeid{\Descendantscompl(5)}  & \Ymat{(5,2)} \!\nodeid{\Descendants(2)} & \Ymat{(5,2)}\! \nodeid{5}   \\
       \end{block}
       \end{blockarray}
    \! = \!
    \begin{blockarray}{ccc} 
       \Descendantscompl(5) & \Descendants(2) & 5 \\
       \begin{block}{[ccc]}
\Zmat{(5,1)}^{\Descendantscompl(5)}  \Ymat{(7,5)}   & \Zmat{(5,1)}^{\Descendants(2)}  \Ymat{(2,5)} & \Zmat{(7,5)}^5 \\
       \end{block}
       \end{blockarray}.
\end{multline*}
Since $\Ymat{(7,5)} = \begin{bmatrix}
                  \Vmat^7_{5,6} \Umat^6_{7,3} \Bmat^3_6 &   \Vmat^7_{6,5} \Umat^6_{7,4} \Bmat^4_6 &   \Vmat^7_{5,6} \Bmat^6_7  &   \Pmat^7_5  
            \end{bmatrix}$, $\Ymat{(2,5)} =\Bmat^2_5$ and   $\UnitHankelBlock{(5,1)}$ should factor into
\begin{displaymath}
    \begin{blockarray}{cc}  \begin{block}{c[c]} 1 & \Qmat^1_5  \\ \end{block} \end{blockarray}  
    \begin{blockarray}{ccc}
    \Descendantscompl(5) &  \Descendants(2) & 5 \\
        \begin{block}{[ccc]}
            \Wmat^5_{1,7} \begin{bmatrix}
                  \Vmat^7_{5,6} \Umat^6_{7,3} \Bmat^3_6 &   \Vmat^7_{6,5} \Umat^6_{7,4} \Bmat^4_6 &   \Vmat^7_{5,6} \Bmat^6_7  &   \Pmat^7_5  
            \end{bmatrix} &  \Vmat^5_{1,2} \Bmat^2_5  & \Pmat^5_1 \\
        \end{block}
    \end{blockarray},
\end{displaymath}
  we see that $\Qmat^1_5  = \Xmat{(5,1)} \nodeid{1} \in \field^{m_1 \times \weight{(5,1)}}$, $\Wmat^5_{1,7} = \Zmat{(5,1)}^{\Descendantscompl(5)} \in \field^{\weight{(5,1)} \times \weight{(7,5)}}$,   $\Vmat^5_{1,2} = \Zmat{(5,1)}^{\Descendants(2)} \in \field^{\weight{(5,1)\times \weight{(2,5)}}}$,  $\Pmat^5_1 = \Zmat{(7,5)}^{5}\in\field^{\weight{(5,1)}\times n_5}$ with $\weight{(5,1)}=\rank \UnitHankelBlock{(5,1)}$. The remaining terms of the TQS representation are obtained in the same way.

\subsection{The general construction algorithm} \label{sec:constructionalg}

The approach taken for the example of the previous section generalizes for a generic TQS representation.  This process is described in \cref{alg:construction}. The process of converting a dense matrix into TQS form consists of two phases: an upsweep and a downsweep phase. In the upsweep phase, one starts at the leaves of the tree and works up towards the root. In this process, all the $\Bmat$'s, $\Cmat$'s, and $\Umat$'s are computed. In the downsweep phase that follows, one starts at the root of the tree and then works down towards the leaves. In this second leg,  the $\Pmat$'s, $\Qmat$'s, $\Wmat$'s, and $\Vmat$'s are computed. The additional `nomenclature' introduced in \Cref{sec:TQSdefinition} reveals that all of the generators are computed exactly once, thus alluding to any inconsistencies that may occur.

\begin{alg}[TQS construction algorithm]  \label{alg:construction} 
Let $\Graph(r)$ be a tree with root node $r\in \NodeSet$ and let $\TQS\in \field^{M\times N}$ be the associated graph-partitioned matrix  with  $M=\sum_{i\in \NodeSet} m_i$ and $N=\sum_{i\in \NodeSet} n_i$. A set of generators for the TQS representation of $T$ is obtained by following the steps outlined below.
    \begin{enumerate}
    \item \textbf{Diagonal stage.} Set $\Dmat^{i}=\TQS\nodeid{i,i}$ for $i\in \NodeSet$.
   \item \textbf{Upsweep stage.} For $l=L, L-1, \ldots, 1$ do the following:
   \begin{enumerate}
       \item For every $i\in \NodeSet_l$ with parent node $j=\Parent(i) \in \NodeSet_{l-1}$ and children $\Children(i)=\lbrace w_1, w_2,\ldots, w_{\alpha}\rbrace
       \subset \NodeSet$, write  $\UnitHankelBlock{(i,j)}=\Fmat{(i,j)}  \Gmat{(i,j)}$, where
       \begin{eqnarray*}
         \Fmat{(i,j)} & := &  \begin{blockarray}{ccccc} \begin{block}{c[cccc]}
            \Descendantscompl(i) &  \Xmat{(w_1,i)}\nodeid{\Descendantscompl(i)} & \cdots & \Xmat{(w_{\alpha},i)}\nodeid{\Descendantscompl(i)} &    \TQS\nodeid{\Descendantscompl(i),i} \\
                 \end{block}  \end{blockarray},  \\
         \Gmat{(i,j)} & := & \begin{blockarray}{cccc}
        \Descendants(w_1) & \cdots & \Descendants(w_{\beta}) & i \\
         \begin{block}{[cccc]}
            \Ymat{(w_1,i)}    &   \\
           & \ddots &  \\
           &      &  \Ymat{(w_{\beta},i)}  \\
           &    &        & \Id \\
           \end{block}
         \end{blockarray}. 
       \end{eqnarray*}
       \item Let $\weight{(i,j)} = \rank \Fmat{(i,j)} =  \rank \UnitHankelBlock{(i,j)} $ and compute the low-rank compression $\Fmat{(i,j)} = \Xmat{(i,j)} \Zmat{(i,j)}$.
       \item Set 
       \begin{displaymath}
           \Bmat^i_{j} :=  \Zmat{(i,j)}^i, \quad \Cmat^j_{i} := \Xmat{(i,j)}\nodeid{j}, \quad \Umat^{i}_{j,w_t}  := \Zmat{(i,j)}^{\Descendants(w_t)}
       \end{displaymath}
       for $t=1,2,\ldots,\alpha$.
       \item Define $\Ymat{(i,j)} =  \Zmat{(i,j)} \Gmat{(i,j)}$ so that $\Xmat{(i,j)} \Ymat{(i,j)}$ is a low-rank factorization for $\UnitHankelBlock{(i,j)}$.
   \end{enumerate}
   \item \textbf{Downsweep stage.} For $l=1, 2, \ldots, L$ do the following:
    \begin{enumerate}
       \item For every $i\in \NodeSet_l$ with parent node  $j=\Parent(i)\in \NodeSet_{l-1}$, grandparent node\footnote{For $l=1$ there will be no grandparent node, in which case the corresponding terms associated with it can be ignored.} $k= \Parent(i;2) \in \NodeSet_{l-2}$, and siblings $\Siblings(i)=\lbrace v_1, v_2,\ldots, v_{\beta}\rbrace \subset \NodeSet$, write $\UnitHankelBlock{(j,i)}=\Fmat{(j,i)}  \Gmat{(j,i)}$, where
       \begin{eqnarray*}
         \Fmat{(j,i)} & := &  \begin{blockarray}{cccccc} \begin{block}{c[ccccc]} \Descendants(i) & \Xmat{(k,j)} \nodeid{\Descendants(i)}  &  \Xmat{(v_1,j)} \nodeid{\Descendants(i)} & \cdots & \Xmat{(v_{\beta},j)} \nodeid{\Descendants(i)}   &   \TQS\nodeid{\Descendants(i),j} \\ \end{block}
        \end{blockarray}  \\
         \Gmat{(j,i)} & := & \begin{blockarray}{ccccc} \Descendantscompl(j) & \Descendants(v_1) & \cdots & \Descendants(v_{\beta}) & j \\ 
         \begin{block}{[ccccc]}  \Ymat{(k,j)}  & & & \\ 
            & \Ymat{(v_1,j)}  &  & \\ 
            &  & \ddots & \\ 
            &  &     & \Ymat{(v_{\beta},j)} \\ 
            &  &     &            & \Id \\
         \end{block}
          \end{blockarray}
       \end{eqnarray*}
       \item Let $\weight{(j,i)} = \rank \Fmat{(j,i)} =  \rank \UnitHankelBlock{(j,i)}$ and compute the low-rank compression $\Fmat{(j,i)} = \Xmat{(j,i)} \Zmat{(j,i)}$.
       \item Set  
       \begin{displaymath}
           \Pmat^j_{i} = \Zmat{(j,i)}^{j},\quad \Wmat^i_{(j,k)} = \Zmat{(j,i)}^{\Descendantscompl(j)} ,\quad 
            \Qmat^i_j = \Xmat{(j,i)}\nodeid{i}, \quad   \Vmat^i_{j,v_t} = \Zmat{(j,i)}^{\Descendants(v_t)} 
       \end{displaymath}
       for $t=1,2,\ldots,\beta$.
       \item Set $\Ymat{(j,i)} =  \Zmat{(j,i)} \Gmat{(j,i)}$ so that $\Xmat{(j,i)} \Ymat{(j,i)}$ is a low-rank factorization for $\UnitHankelBlock{(j,i)}$.
   \end{enumerate}
\end{enumerate}
\end{alg}
 We remark that \Cref{alg:construction} is not the only approach for constructing a TQS matrix. There exists some flexibility in algorithmic design choices that could be optimized for parallelism and memory consumption. A more detailed analysis goes outside the scope of this paper. For now, we stay contented that \Cref{alg:construction} presents a valid construction/realization algorithm.

\subsection{Numerical tests} To further validate the construction algorithm for correctness, we have implemented \Cref{alg:construction} in the Julia language\footnote{This code is made available in \href{https://github.com/nithingovindarajan/TQSmatrices}{https://github.com/nithingovindarajan/TQSmatrices}.}. \Cref{tab:main_table} showcases the numerical results obtained with the implemented algorithm for two experiments. In the first experiment, \Cref{alg:construction} is applied to reconstruct minimal TQS representations of randomly generated TQS matrices for the tree graphs $\Graph_a(4)$, $\Graph_b(4)$, $\Graph_c(5)$, and $\Graph_d(7)$. The TQS matrices are first converted into dense matrices after which \Cref{alg:construction} is applied to reconstruct the TQS matrix. In the second experiment, \Cref{alg:construction} is used to construct TQS representations for the inverse of the matrices $\TQS_k$, $k=1,2,3,\ldots$ from \Cref{sec:sparsematrices}. The parameters used to generate $\TQS_k$ are again chosen randomly. Using the adjacency graph as the corresponding tree, it follows from \Cref{prop:TQSinverse} that $\TQS^{-1}_k$ admit a scalar TQS representation, i.e., $\weight{e} = 1$ $\forall e \in \EdgeSet$. Our numerical experiment also confirmed this property. \Cref{table1} and \Cref{table2} show the relative 2-norm error of the constructed TQS matrices (with respect to the original dense matrix) for experiments 1 and 2, respectively. The results suggest that \Cref{alg:construction} can generate TQS approximations up to machine precision accuracy.

\begin{table}[h]
    \centering
    \begin{minipage}{.45\textwidth}
        \centering
  \begin{tabular}{ccc}
        \hline
        \multirow{2}{*}{ } & \multicolumn{2}{c}{\textit{rel. error}} \\
        \cline{2-3}
         & mean & std \\
        \hline
 $\Graph_a(4)$ & 2.2431e-15 & 2.0166e-15 \\
 $\Graph_b(4)$ & 7.6436e-16 & 1.4924e-16 \\
 $\Graph_c(5)$  & 6.5400e-16 & 2.4239e-16 \\
$\Graph_d(7)$  & 4.0584e-15 & 2.8004e-15\\
        \hline
    \end{tabular}
        \subcaption{Computed error statistics for the TQS reconstruction of randomly generated TQS matrices for the tree graphs $\Graph_a(4)$, $\Graph_b(4)$, $\Graph_c(5)$, and $\Graph_d(7)$. } \label{table1}
    \end{minipage}%
    \hfill
    \begin{minipage}{.45\textwidth}
        \centering
  
    \begin{tabular}{ccc}
        \hline
        \multirow{2}{*}{$k$} & \multicolumn{2}{c}{\textit{rel. error}} \\
        \cline{2-3}
         & mean & std \\
        \hline
 1  & 0.0        &  0.0 \\
 2 & 1.41199e-16 &  4.87742e-17 \\
 3 & 2.78289e-16 & 8.81678e-17 \\
 4 & 3.6505e-16  & 1.66578e-16 \\
 5 & 4.71926e-16 & 2.48319e-16 \\
 6 &  6.95086e-16 & 2.30473e-16 \\
 7 & 5.68185e-16 & 1.56358e-16 \\
 8 & 1.03728e-15 &  3.58987e-16 \\
 9 & 1.1168e-15  & 4.68181e-16 \\
        \hline
    \end{tabular}
        \subcaption{Computed error statistics for the TQS construction of the  inverse of the matrices $\TQS_k$, $k=1,2,3,\ldots$ from \Cref{sec:sparsematrices}.} \label{table2}
    \end{minipage}
    \caption{Numerical results obtained with \Cref{alg:construction}. The mean and standard deviation of the relative 2-norm error of the constructed TQS matrix (w.r.t. the original dense matrix) is computed. The statistics are computed from 10 random trials.}
    \label{tab:main_table}
\end{table}

\subsection{Proof of \Cref{thm:universality}} \label{sec:proofuniversality}
We are now ready to prove \Cref{thm:universality}.
\begin{proof}[Proof of \Cref{thm:universality}]
    \Cref{alg:construction} presents a constructive proof for the existence of a TQS representation with rank-profile $\lbrace \weight{e} := \rank \UnitHankelBlock{e} \rbrace_{ e \in \EdgeSet}$.  The only thing left is to show that the TQS representation produced by \Cref{alg:construction} is minimal. This can be verified by setting up a contradiction.  Suppose there exists a TQS representation with rank-profile $\lbrace \rho'_{e} \rbrace_{e\in\EdgeSet}$  and $\rho'_{e} < \weight{e}$  for some edge $e\in\EdgeSet$, then the unit Hankel $\UnitHankelBlock{e}$  admits a low-rank factorization of rank $\rho'_{e} < \weight{e}$, which is not possible.
\end{proof}

\section{TQS linear systems} \label{sec:TQSsystems}
In this section, we examine linear systems $\TQS \vect{x} = \vect{b}$, where $\TQS \in \field^{M\times N}$is a TQS matrix on $\Graph=(\NodeSet,\EdgeSet)$  of dimensions  $M=\sum_{i\in \NodeSet} m_i$ by $N=\sum_{i\in \NodeSet} n_i$. By conformally partitioning $\vect{b}
\in \field^{M}$ and  $\vect{x}
\in \field^{N}$ into sub-vectors  $\vect{b}_i \in \field^{m_i}$ and $\vect{x}_i \in \field^{n_i}$, respectively, we first describe in \Cref{sec:matvec}  how the matrix-vector product $\vect{b}= \TQS \vect{x}$ is evaluated efficiently. Then, in \Cref{sec:linsys}, we proceed and use the relations derived for the matrix-vector product to formulate an efficient method to solve  $\TQS \vect{x} = \vect{b}$ is solved for $\vect{x}$ given $\vect{b}$.

\subsection{Evaluation of matrix-vector product} \label{sec:matvec}

The block entries of a TQS matrix originate from evolving a dynamical system over a graph. In the special case of an SSS matrix, the lower and upper triangular parts of the matrix are the result of running a causal and anti-causal linear time-variant dynamical system. This is exploited in the matrix-vector product. For the general case, the dynamics evolve over a tree, and it becomes much harder to provide a simple characterization. Nonetheless, formulating the ``state-space'' equations shall produce a fast matrix-vector product algorithm for the general case as well.  Recall that each edge $(i,j)\in \EdgeSet$ of the acyclic connected graph $\Graph$ a state vector $\state{(i,j)}\in \field^{\weight{(i,j)}}$.  The transition maps \eqref{eq:transitionmaps} yield the state equations
\begin{equation}
      \state{(i,j)} = \sum_{w \in \kNeighbor(i) \setminus \{ j \}}   \Trans^{i}_{j,w} \state{(w,i)}       +     \Inp^{i}_j \vect{x}_i, \qquad  (i,j) \in \EdgeSet, \label{eq:staterel}
\end{equation}
along with the output equations
\begin{equation}
      \vect{b}_j = \sum_{i \in \kNeighbor(j)}  \Out^j_i \state{(i,j)}   + \Dmat^j \vect{x}_i,   \qquad   j \in \NodeSet. \label{eq:outputrel}
\end{equation}
By picking a root node $r\in \NodeSet$, the state and output equations for the corresponding tree $\Graph(r)$ may be further refined to 
\begin{equation}
      \state{(i,j)} = \begin{cases} \displaystyle \sum_{w \in \Children(i)}   \Umat^{i}_{j,w} \state{(w,i)}       +     \Bmat^{i}_j \vect{x}_i, \qquad  & j = \Parent(i) \\
          \Wmat^{i}_{j,\Parent(i)} \state{(\Parent(i),i)}  +  \displaystyle   \sum_{w \in \Siblings(j)}   \Vmat^{i}_{j,w} \state{(w,i)}       +     \Pmat^{i}_j \vect{x}_i, \qquad           &  i = \Parent(j)
      \end{cases} \label{eq:statefortree}
\end{equation}
and 
\begin{equation}
   \vect{b}_j =   \Qmat^j_{\Parent(j)} \state{(\Parent(j),j)}  + \sum_{i \in \Children(j)}  \Bmat^j_i \state{(i,j)}   + \Dmat^j \vect{x}_i.
\end{equation}
A careful examination of \eqref{eq:statefortree} reveals a natural causal ordering on the state variables. Specifically, for a leaf $i\in\Leaves(\Graph)$, the state equations are simply $\state{(i,j)} = \Bmat^i_{j} \vect{x}_i$.  One may thus start with computing the state vectors at the leaves and then work up toward the interiors of the graph. Once the root node is reached, the reverse process can be initiated by flowing outwards towards the leaves. \Cref{alg:matvec} exactly describes this process. The TQS matrix-vector product involves
\begin{displaymath}
   \mathcal{O} \left( \sum_{i\in\NodeSet} \left( m_i + \sum_{j \in \kNeighbor(i)} r_{(i,j)}  \right)  \left( n_i + \sum_{j \in \kNeighbor(i)} r_{(j,i)}  \right)  \right)
\end{displaymath}
floating point operations (flops). In particular, the complexity becomes a linear-time w.r.t. the matrix dimensions if the properties in \eqref{eq:linearcomplexity} are applicable.

\begin{alg}[TQS matrix-vector product] \label{alg:matvec}
   Given a TQS matrix $\TQS \in \field^{M\times N}$ on $\Graph(r)$  of dimensions  $M=\sum_{i\in \NodeSet} m_i$ by $N=\sum_{i\in \NodeSet} n_i$, the matrix-vector product $\vect{b}_{i} = \sum_{j\in \NodeSet} \TQS\nodeid{i,j} \vect{x}_j$ for $i\in\NodeSet$ is obtained by following the steps outlined below.  
   \begin{enumerate}
    \item \textbf{Diagonal stage.} Initialize $\vect{b}_{i} =\Dmat^i \vect{x}_i$ for $i\in \NodeSet$.
   \item \textbf{Upsweep stage.} For $l=L, L-1, \ldots, 1$ do the following:
   \begin{enumerate}
       \item For every $i\in \NodeSet_l$ with parent node $j=\Parent(i) \in \NodeSet_{l-1}$ and children $\Children(i)=\lbrace w_1, w_2,\ldots, w_{\alpha}\rbrace
       \subset \NodeSet$, evaluate
       \begin{displaymath}
          \state{(i,j)} =  \sum^{\alpha}_{p=1}   \Umat^{i}_{j,w_{p}} \state{(w_p,i)}       +     \Bmat^{i}_j \vect{x}_i.
       \end{displaymath}
      \item Update 
      \begin{displaymath}
          \vect{b}_j \leftarrow \vect{b}_j +  \Cmat^j_i \state{(i,j)}.
      \end{displaymath}
   \end{enumerate}
   \item \textbf{Downsweep stage.} For $l=1, 2, \ldots, L$ do the following:
    \begin{enumerate}
       \item For every $i\in \NodeSet_l$ with parent node  $j=\Parent(i)\in \NodeSet_{l-1}$, grandparent node $k= \Parent(i;2) \in \NodeSet_{l-2}$, and siblings $\Siblings(i)=\lbrace v_1, v_2,\ldots, v_{\beta}\rbrace \subset \NodeSet$, evaluate
        \begin{displaymath}
          \state{(j,i)} =  \Wmat^{j}_{i,k} \state{(k,j)}  +  \sum^{\beta}_{p=1}   \Vmat^{j}_{i,w_{p}} \state{(w_p,i)}       +     \Bmat^{j}_i \vect{x}_j.
       \end{displaymath}
       \item Update 
      \begin{displaymath}
          \vect{b}_i \leftarrow \vect{b}_i +  \Qmat^i_j \state{(j,i)}.
      \end{displaymath}
   \end{enumerate}
\end{enumerate}
\end{alg}

\subsection{Solving linear systems} \label{sec:linsys}
The linear system $\TQS \vect{x} = \vect{b}$ is efficiently solved for $\vect{x}$ given $\vect{b}$ by using the same reasoning introduced in \cite{chandrasekaran2007fast}. By treating  $\{\state{e} \}_{e\in\EdgeSet}$ and $\{ \vect{x}_i\}_{i\in\NodeSet}$ as unknowns, and $\{ \vect{b}_i\}_{i\in\NodeSet}$ as knowns, \eqref{eq:staterel} and \eqref{eq:outputrel} collectively yield the sparse linear system
\begin{subequations}
    \begin{eqnarray}
      \state{(i,j)}   - \displaystyle \sum_{w \in \kNeighbor(i) \setminus \{ j \}}   \Trans^{i}_{j,w} \state{(w,i)}       -     \Inp^{i}_j \vect{x}_i &  = & 0,  \qquad  (i,j) \in \EdgeSet \\
     \displaystyle \sum_{i \in \kNeighbor(j)}  \Out^j_i \state{(i,j)}   + \Dmat^j \vect{x}_i &  =&   \vect{b}_j, \qquad j \in \NodeSet.
\end{eqnarray} \label{eq:sparsesys}
\end{subequations}
Particularly, if we let $\kNeighbor(j) = \{ i_1,i_2,\ldots, i_p \}$ and define
\begin{displaymath}
    \thetavar_j := \begin{bmatrix}
    \vect{x}_j \\  \state{(i_1,j)} \\ \state{(i_2,j)} \\  \vdots \\  \state{(i_p,j)} 
\end{bmatrix}, \qquad \betavar_j = \begin{bmatrix}
    \vect{b}_j \\  0 \\ 0 \\ \vdots \\  0
\end{bmatrix}
\end{displaymath}
the adjecency graph of the matrix expression $\Xi \thetavar = \betavar$ that describes \eqref{eq:sparsesys}  coincides with $\Graph$. That is, $\Xi\nodeid{i,j} \ne 0$ if and only if $(i,j)\in \EdgeSet$. Acyclic graphs are chordal graphs and have a perfect elimination order without any fill-in. Thus, \eqref{eq:sparsesys} may be efficiently solved with any standard (block-)sparse solver. Solving \eqref{eq:sparsesys} involves roughly
\begin{displaymath}
   \mathcal{O} \left( \sum_{i\in\NodeSet} \left( n_i + \sum_{j \in \kNeighbor(i)} r_{(i,j)}  \right)^3  \right)
\end{displaymath}
floating point operations (flops). In particular, the complexity becomes linear w.r.t. the matrix dimensions if the properties in \eqref{eq:linearcomplexity} are applicable.

\section{Conclusions {and future work}} \label{sec:conclusions}
We introduced a new class of representations for rank-structured matrices called tree quasi-separable (TQS) matrices. It was shown that TQS matrices unify and generalize SSS and HSS matrices. Furthermore, by deriving an explicit construction algorithm, we characterized the properties of a minimal TQS representation for a given tree graph-partitioned matrix. Subsequently, we showed that TQS inherits many of the well-known properties of SSS and HSS matrices concerning matrix-vector multiply, matrix-matrix multiply, matrix-matrix addition, and inversion.

Future work will be geared towards the efficient implementation and generalization of many algorithms associated with SSS and HSS matrices. Specifically, in a future paper, we shall derive expressions for the LU factorization and (pseudo-)inverse of TQS matrices.  The potential applications of TQS (and the greater flexibility that is offered by them) will also be explored. Finally, we note that the results associated with TQS may form an essential building block for constructing representations associated with more general 
GIRS matrices (see \cite{girstsstalk}).

\bibliographystyle{siamplain}
\bibliography{references}

\begin{thebibliography}{10}

\bibitem{ambikasaran2013mathcal}
{\sc S.~Ambikasaran and E.~Darve}, {\em An $\mathcal{O} (n \log n)$ fast direct solver for partial hierarchically semi-separable matrices}, Journal of Scientific Computing, 57 (2013), pp.~477--501.

\bibitem{aminfar2016fast}
{\sc A.~Aminfar, S.~Ambikasaran, and E.~Darve}, {\em A fast block low-rank dense solver with applications to finite-element matrices}, Journal of Computational Physics, 304 (2016), pp.~170--188.

\bibitem{aurentz2015fast}
{\sc J.~L. Aurentz, T.~Mach, R.~Vandebril, and D.~S. Watkins}, {\em Fast and backward stable computation of roots of polynomials}, SIAM Journal on Matrix Analysis and Applications, 36 (2015), pp.~942--973.

\bibitem{bella2008computations}
{\sc T.~Bella, Y.~Eidelman, I.~Gohberg, and V.~Olshevsky}, {\em Computations with quasiseparable polynomials and matrices}, Theoretical Computer Science, 409 (2008), pp.~158--179.

\bibitem{borm2003introduction}
{\sc S.~B{\"o}rm, L.~Grasedyck, and W.~Hackbusch}, {\em Introduction to hierarchical matrices with applications}, Engineering Analysis with Boundary Elements, 27 (2003), pp.~405--422.

\bibitem{chandrasekaran2007fast}
{\sc S.~Chandrasekaran, P.~Dewilde, M.~Gu, W.~Lyons, and T.~Pals}, {\em A fast solver for hss representations via sparse matrices}, SIAM Journal on Matrix Analysis and Applications, 29 (2007), pp.~67--81.

\bibitem{chandrasekaran2005some}
{\sc S.~Chandrasekaran, P.~Dewilde, M.~Gu, T.~Pals, X.~Sun, A.-J. van~der Veen, and D.~White}, {\em Some fast algorithms for sequentially semiseparable representations}, SIAM Journal on Matrix Analysis and Applications, 27 (2005), pp.~341--364.

\bibitem{chandrasekaran2002fast}
{\sc S.~Chandrasekaran, P.~Dewilde, M.~Gu, T.~Pals, and A.-J. van~der Veen}, {\em Fast stable solver for sequentially semi-separable linear systems of equations}, in International Conference on High-Performance Computing, Springer, 2002, pp.~545--554.

\bibitem{chandrasekaran2010numerical}
{\sc S.~Chandrasekaran, P.~Dewilde, M.~Gu, and N.~Somasunderam}, {\em On the numerical rank of the off-diagonal blocks of schur complements of discretized elliptic pdes}, SIAM Journal on Matrix Analysis and Applications, 31 (2010), pp.~2261--2290.

\bibitem{chandrasekaran2019graph}
{\sc S.~Chandrasekaran, E.~N. Epperly, and N.~Govindarajan}, {\em Graph-induced rank structures and their representations}, arXiv preprint arXiv:1911.05858,  (2019).

\bibitem{girstsstalk}
{\sc S.~Chandrasekaran and N.~Govindarajan}, {\em Graph induced rank structured matrices and their representations}.
\newblock talk at CAM23 - Computational and Applied Mathematics, Selva di Fasano, Italy, 2023.

\bibitem{shivnithinabhe}
{\sc S.~Chandrasekaran, N.~Govindarajan, and A.~Rajagopal}, {\em Fast algorithms for displacement and low-rank structured matrices}, in Proc. of the 2018 ACM International Symposium on Symbolic and Algebraic Computation, ISSAC '18, {New York, NY, USA}, 2018, Association for Computing Machinery, p.~17–22.

\bibitem{chandrasekaran2006fast}
{\sc S.~Chandrasekaran, M.~Gu, and T.~Pals}, {\em A fast ulv decomposition solver for hierarchically semiseparable representations}, SIAM Journal on Matrix Analysis and Applications, 28 (2006), pp.~603--622.

\bibitem{chandrasekaran2008superfast}
{\sc S.~Chandrasekaran, M.~Gu, X.~Sun, J.~Xia, and J.~Zhu}, {\em A superfast algorithm for toeplitz systems of linear equations}, SIAM Journal on Matrix Analysis and Applications, 29 (2008), pp.~1247--1266.

\bibitem{dewilde1998time}
{\sc P.~Dewilde and A.-J. Van~der Veen}, {\em Time-Varying Systems and Computations}, Springer Science \& Business Media, Dordrecht, The Netherlands, 1998.

\bibitem{eidelman1999new}
{\sc Y.~Eidelman and I.~Gohberg}, {\em On a new class of structured matrices}, Integral Equations and Operator Theory, 34 (1999), pp.~293--324.

\bibitem{greengard1987fast}
{\sc L.~Greengard and V.~Rokhlin}, {\em A fast algorithm for particle simulations}, Journal of Computational Physics, 73 (1987), pp.~325--348.

\bibitem{hackbusch2015hierarchical}
{\sc W.~Hackbusch}, {\em Hierarchical matrices: algorithms and analysis}, vol.~49, Springer, 2015.

\bibitem{hackbusch2002data}
{\sc W.~Hackbusch and S.~B{\"o}rm}, {\em Data-sparse approximation by adaptive $\mathcal{H}^2$-matrices}, Computing, 69 (2002), pp.~1--35.

\bibitem{rokhlin1985rapid}
{\sc V.~Rokhlin}, {\em Rapid solution of integral equations of classical potential theory}, Journal of computational physics, 60 (1985), pp.~187--207.

\bibitem{vandebril2008matrix}
{\sc R.~Vandebril, M.~Van~Barel, and N.~Mastronardi}, {\em Matrix computations and semiseparable matrices: linear systems}, vol.~1, JHU Press, 2008.

\bibitem{verhaegen2022data}
{\sc M.~Verhaegen, C.~Yu, and B.~Sinquin}, {\em Data-Driven Identification of Networks of Dynamic Systems}, Cambridge University Press, 2022.

\bibitem{wang2013efficient}
{\sc S.~Wang, X.~S. Li, J.~Xia, Y.~Situ, and M.~V. De~Hoop}, {\em Efficient scalable algorithms for solving dense linear systems with hierarchically semiseparable structures}, SIAM J. Sci. Comput., 35 (2013), pp.~C519--C544.

\bibitem{xia2012complexity}
{\sc J.~Xia}, {\em On the complexity of some hierarchical structured matrix algorithms}, SIAM J. Matrix Anal. Appl., 33 (2012), pp.~388--410.

\bibitem{xia2012robust}
{\sc J.~Xia}, {\em Robust and efficient multifrontal solver for large discretized pdes}, High-Performance Scientific Computing: Algorithms and Applications,  (2012), pp.~199--217.

\bibitem{xia2010fast}
{\sc J.~Xia, S.~Chandrasekaran, M.~Gu, and X.~S. Li}, {\em Fast algorithms for hierarchically semiseparable matrices}, Numerical Linear Algebra with Applications, 17 (2010), pp.~953--976.

\bibitem{xia2012superfast}
{\sc J.~Xia, Y.~Xi, and M.~Gu}, {\em A superfast structured solver for toeplitz linear systems via randomized sampling}, SIAM J. Matrix Anal. Appl., 33 (2012), pp.~837--858.

\end{thebibliography}
\end{document}